\documentclass[11pt]{article}

\usepackage{amsthm}
\usepackage{amsmath}
\usepackage{amssymb}
\usepackage{amsmath, amsthm, bbm}
\usepackage{framed, fullpage}

\usepackage{mathrsfs}

\usepackage[backref,colorlinks,citecolor=blue,bookmarks=true]{hyperref}

\newtheorem*{theo*}{Theorem}
\newtheorem{theo}{Theorem}

\newtheorem{lemma}{Lemma}[section]
\newtheorem{coro}[theo]{Corollary}
\newtheorem{definition}[lemma]{Definition}
\newtheorem{prop}[lemma]{Proposition}

\newtheorem{claim}[lemma]{Claim}
\newtheorem{facts}{Fact}
\newtheorem{fact}[facts]{Fact}

\newtheorem{remark}[lemma]{Remark}

\makeatletter
\renewenvironment{proof}[1][\proofname]
{\par\pushQED{\qed}
	\normalfont\topsep6\p@\@plus6\p@\relax\trivlist
	\item[\hskip\labelsep\bfseries#1\@addpunct{.}]
	\ignorespaces}
{\popQED\endtrivlist\@endpefalse}
\makeatother

\newcommand{\F}{\mathcal{F}}
\renewcommand{\S}{\mathcal{S}}
\newcommand{\N}{\mathbb{N}}

\newcommand{\Ex}{\mathbb{E}}

\newcommand{\ith}{i\text{-th}}

\newcommand{\norm}[1]{\left\lVert #1 \right\rVert}
\newcommand{\floor}[1]{\left\lfloor{#1}\right\rfloor}
\newcommand{\ceil}[1]{\left\lceil #1 \right\rceil}

\renewcommand{\P}{\mathcal{P}}
\renewcommand{\Pr}{\mathbb{P}}
\newcommand{\Q}{\mathcal{Q}}
\newcommand{\R}{\mathcal{R}}
\newcommand{\E}{\mathcal{E}}
\newcommand{\A}{\mathcal{A}}
\newcommand{\B}{\mathcal{B}}
\newcommand{\C}{\mathcal{C}}
\newcommand{\T}{\mathcal{T}}
\newcommand{\V}{\mathcal{V}}

\DeclareMathOperator{\poly}{poly}

\DeclareMathOperator{\Csub}{\Psi}

\newcommand{\tower}[2]{\twr({#1},\,{#2})}

\DeclareMathOperator{\twr}{T}
\DeclareMathOperator{\wow}{W}

\renewcommand{\a}{\alpha}
\renewcommand{\b}{\beta}
\renewcommand{\d}{\delta}
\newcommand{\g}{\gamma}

\renewcommand{\L}{\mathcal{L}}
\newcommand{\sub}{\subseteq}
\newcommand{\sm}{\setminus}

\newcommand{\e}{\epsilon}

\newcommand{\Z}{\mathcal{Z}}
\renewcommand{\l}{\ell}
\newcommand{\HH}{\mathcal{H}}

\newcommand{\X}{\mathcal{X}}
\newcommand{\Y}{\mathcal{Y}}

\newcommand{\D}{\mathcal{D}}

\newcommand{\K}{\mathcal{K}}
\newcommand{\G}{\mathcal{G}}
\newcommand{\I}{\mathcal{I}}

\newcommand{\h}[1]{\widehat{#1}}
\renewcommand{\c}[1]{{#1}}

\newcommand{\Aside}{\mathbf{A}}
\newcommand{\Bside}{\mathbf{B}}
\newcommand{\Cside}{\mathbf{C}}

\newcommand{\Lside}{\mathbf{L}}
\newcommand{\Rside}{\mathbf{R}}

\newcommand{\Hy}[1]{H}

\DeclareMathOperator{\Exp}{\mathbb{E}}
\DeclareMathOperator{\codeg}{codeg}
\DeclareMathOperator{\Var}{Var}

\DeclareMathOperator{\dist}{dist}

\DeclareMathOperator{\Ack}{Ack}

\DeclareMathOperator{\Cross}{Cross}

\renewcommand{\k}{\kappa}

\newcommand{\Vside}{\mathbf{V}}

\DeclareMathOperator{\under}{U}

\date{}

\title{A Tight Bound for Hypergraph Regularity II}
\author{Guy Moshkovitz\thanks{School of Mathematics, Tel Aviv University, Tel Aviv 69978, Israel.  Email: {\tt guymosko@tau.ac.il}. Supported in part by ERC Starting Grant 633509.}
\and Asaf Shapira\thanks{School of Mathematics, Tel Aviv University, Tel Aviv 69978, Israel.
Email: {\tt asafico$@$tau.ac.il}. Supported in part by ISF Grant 1028/16 and ERC Starting Grant 633509.}}
\begin{document}

\maketitle
\begin{abstract}
The hypergraph regularity lemma -- the extension of Szemer\'edi's graph regularity lemma to the setting of $k$-uniform hypergraphs -- is one of the most celebrated combinatorial results obtained in the past decade.
By now there are several (very different) proofs of this lemma, obtained by Gowers, by Nagle-R\"odl-Schacht-Skokan and by Tao.
Unfortunately, what all these proofs have in common is that they yield regular partitions whose order is given by the $k$-th Ackermann function.

In a recent paper we have shown that these bounds are unavoidable for $3$-uniform hypergraphs. 
In this paper we extend this result by showing that such Ackermann-type bounds are unavoidable for every $k \ge 2$, thus confirming a prediction of Tao.
\end{abstract}

\section{Introduction}

One of the most surprising applications of Szemer\'edi's regularity lemma is the proof of Roth's theorem~\cite{Roth54}. 
This ingenuous proof, due to Ruzsa and Szemer\'edi~\cite{RuzsaSz76}, implicitly relies on what is now known as the
{\em triangle removal lemma}. Erd\H{o}s, Frankl and R\"odl~\cite{ErdosFrRo86} asked if this lemma can be extended
to the setting of $k$-uniform hypergraphs, and Frankl and R\"odl~\cite{FrankRo02} observed that such a result would allow one to extend the Ruzsa--Szemer\'edi~\cite{RuzsaSz76} argument and thus obtain an alternative proof of Szemer\'edi's theorem~\cite{Szemeredi75} for progressions of arbitrary length (see also~\cite{Solymosi04}). Frankl and R\"odl~\cite{FrankRo02} further initiated a programme for proving such a hypergraph removal lemma
via a {\em hypergraph regularity lemma} and proved such a lemma for $3$-uniform hypergraphs. This task was completed only 10 years later when R\"odl, Skokan, Nagle and Schacht~\cite{NagleRoSc06, RodlSk04} and independently Gowers~\cite{Gowers07} obtained regularity lemmas $k$-uniform hypergraphs (from now on we will use $k$-graphs instead of $k$-uniform hypergraphs).
Shortly after, Tao~\cite{Tao06} and R\"odl and Schacht~\cite{RodlSc07,RodlSc07-B} obtained two more versions of the lemma.

The above-mentioned variants of the hypergraph regularity lemma relied on four different notions of {\em quasi-randomness}, which
are not known to be equivalent; see R\"odl's recent ICM survey~\cite{Rodl14} and our recent paper \cite{MS3} for more on this. 
What all of these proofs {\em do} have in common however, is that they supply only Ackermann-type bounds for the size of a regular partition. More precisely,
if we let $\Ack_1(x)=2^x$ and then define $\Ack_k(x)$ to be the $x$-times iterated\footnote{So $\Ack_2(x)$ is a tower of exponents of height $x$, i.e.\ the tower function, $\Ack_3(x)$ is the so-called wowzer function, etc.} version of $\Ack_{k-1}$, then all the above proofs guarantee to produce a regular partition of a $k$-graph whose order can be bounded from above by an $\Ack_k$-type function.
Gowers~\cite{Gowers97} famously proved that $\Ack_2$-type upper bounds for graph regularity
are unavoidable. Tao~\cite{Tao06-h} predicted that Gowers's result can be extended to the setting
of $k$-graphs, that is, that $\Ack_k$-type bounds are unavoidable for the $k$-graph regularity lemma.

Until very recently, no analogue of Gowers's lower bound was known for any $k>2$.
In a recent paper \cite{MS3} we obtained such a result for the $3$-graph regularity lemma.
We refer the reader to Section $1$ of \cite{MS3} for a thorough discussion of this result
and some of the key ideas behind it. 
In the present paper we extend this result to arbitrary $k \geq 3$, thus conforming
Tao's prediction~\cite{Tao06-h}. Our main result can be informally stated as follows.


\begin{theo}{\bf[Main result, informal statement]}\label{thm:main-informal}
The following holds for every $k\geq 2$: every regularity lemma for $k$-graphs satisfying some
mild conditions can only guarantee to produce partitions of size bounded by an $\Ack_k$-type function.
\end{theo}

Our main result, stated formally as Theorem~\ref{theo:main}, establishes an $\Ack_k$-type lower bound
for \emph{$\langle \d \rangle$-regularity} of $k$-graphs, which is a new notion we first introduced in~\cite{MS3}. 
The main advantage of this notion is threefold:
$(i)$ It is much simpler to state compared to all other notions of $k$-graph regularity.
$(ii)$ It is weak enough to allow one to  induct on $k$, that is, to use lower bounds for $k$-graph regularity in order to obtain lower bounds for $(k+1)$-regularity\footnote{See Section $1$ of \cite{MS3} for a more detailed discussion on this aspect of the proof.}.
$(iii)$ All known notions of regularity appear to be stronger than $\langle \d \rangle$-regularity, so a lower bound
for $\langle \d \rangle$-regularity gives a lower bound for such lemmas, that is, for any lemma
whose requirements/guarantees imply those that are needed in order to satisfy $\langle \d \rangle$-regularity.

We will demonstrate the effectiveness of item~$(iii)$ above by deriving from Theorem~\ref{thm:main-informal} a lower bound for 
the $k$-graph regularity lemma of R\"odl--Schacht~\cite{RodlSc07}.

\begin{coro}[Lower bound for $k$-graph regularity]\label{coro:RS-LB}
For every $k\geq 2$, there is an $\Ack_k$-type lower bound for the $k$-graph regularity lemma of R\"odl--Schacht~\cite{RodlSc07}.
\end{coro}
As we discuss at the beginning of Section~\ref{sec:coro}, the lower bound stated in Corollary~\ref{coro:RS-LB} holds even for
a very weak/special case of the $k$-graph regularity lemma of~\cite{RodlSc07}.

As Theorem \ref{thm:main-informal} establishes a lower bound for
$\langle \d \rangle$-regularity, it is natural to ask if this notion is in fact equivalent to other notions. In particular,  is this notion strong enough for ``counting'', that is,
for proving the hypergraph removal lemma, which was one of the main reasons for developing
the hypergraph regularity lemma? Our final result (see Proposition \ref{claim:example} for the formal statement) 
answers both questions negatively. 
This of course makes our lower bound even stronger as it already applies to a very weak notion of regularity.

\begin{prop}{\bf[Informal statement]}\label{prop:counter}
$\langle \d \rangle$-regularity is not strong enough even for proving the graph triangle removal lemma.
\end{prop}

\subsection{Paper overview}

Broadly speaking, Section \ref{sec:define} serves as the technical introduction to this paper, while
Section \ref{sec:LB} contains the main technical proofs.
More concretely, in Section \ref{sec:define} we will first define the new notion of $k$-graph regularity, which we term $\langle \d \rangle$-regularity,
for which we will prove our main lower bound. We will then give the formal version of Theorem~\ref{thm:main-informal} (see Theorem~\ref{theo:main}).
This will be followed by the statement of the main technical result we will use in this paper, Theorem~\ref{theo:core},
and an overview of how this technical result is used in the proof of Theorem~\ref{theo:main}.
The proof of Theorem~\ref{theo:main} appears in Section~\ref{sec:LB}. 
In Section \ref{sec:coro} we describe how Theorem~\ref{theo:main} can be used in order to prove Corollary~\ref{coro:RS-LB}, thus establishing tight
$\Ack_k$-type lower bounds for a concrete version of the hypergraph regularity lemma.
Since at its core, the proof of Corollary~\ref{coro:RS-LB} is very similar to the
way we derived lower bounds for concrete regularity lemmas for $3$-uniform hypergraphs in \cite{MS3},
just with a much more elaborate set of notations (due to having to deal with arbitrary $k$),
we decided to put the proof of some technical claims only in the appendix of the Arxiv version of this paper.
Finally, in Section~\ref{sec:example} we prove Proposition \ref{prop:counter} by describing an example showing that even in the setting of graphs,
$\langle \d \rangle$-regularity is strictly weaker than the usual notion of graph regularity, as it does not allow one even to count triangles.

\paragraph{How is this paper related to~\cite{MS3}:}
For the reader's convenience we explain how this paper differs from~\cite{MS3}, in which we prove Theorem~\ref{theo:main} for $k=3$.
First, the definitions given in Section~\ref{sec:define} when specialized to $k=3$ are the same notions used in (Section~2 of)~\cite{MS3}.
The heart of the proof of Theorem~\ref{theo:main} is given by Lemma~\ref{lemma:ind-k} which is proved by induction on $k$.
The (base) case $k=2$ follows easily from Lemma~\ref{theo:core} which was proved in~\cite{MS3}.
Hence, the heart of the matter is the proof of Lemma~\ref{lemma:ind-k} by induction on $k$. Within this framework,
the argument given in~\cite{MS3} 
is precisely the deduction of Lemma~\ref{lemma:ind-k} for $k=3$ from the case $k=2$.
Hence, the reader interested in seeing the inductive proof of Lemma~\ref{lemma:ind-k} ``in action''---without the clutter
caused by the complicated definitions related to $k$-graphs---is advised to check~\cite{MS3}.

\section{$\langle \d \rangle$-regularity and Proof Overview}\label{sec:define}

\subsection{Preliminary definitions}\label{subsec:preliminaries}


%

Before giving the definition of $\langle \d \rangle$-regularity, let us start with some standard definitions regarding partitions of hypergraphs.
Formally, a \emph{$k$-graph} is a pair $H=(V,E)$, where $V=V(H)$ is the vertex set and $E=E(H) \sub \binom{V}{k}$ is the edge set of $H$.
The number of edges of $H$ is denoted $e(H)$ (i.e., $e(H)=|E|$).
We denote by $K^k_\ell$ the complete $\ell$-vertex $k$-graph (i.e., containing  all possible $\binom{\ell}{k}$ edges).
The $k$-graph $H$ is \emph{$\ell$-partite} $(\ell \ge k)$ on (disjoint) vertex classes $(V_1,\ldots,V_\ell)$ if every edge of $H$ has at most one vertex from each $V_i$.
We denote by $H[V'_1,\ldots,V'_\ell]$ the $\ell$-partite $k$-graph induced on vertex subsets $V_1' \sub V_1,\ldots,V_\ell'\sub V_\ell$; that is, $H[V'_1,\ldots,V'_\ell]=((V_1',\ldots,V_\ell'),\, \{e \in E(H) \,\vert\, \forall i \colon e \cap V_i \in V_i'\})$.
The \emph{density} of a $k$-partite $k$-graph $H$ is $e(H)/\prod_{i=1}^k |V_i|$.
The set of edges of $G$ between disjoint vertex subsets $A$ and $B$ is denoted by $E_G(A,B)$; the density of $G$ between $A$ and $B$ is denoted by $d_G(A,B)=e_G(A,B)/|A||B|$, where $e_G(A,B)=|E_G(A,B)|$. We use $d(A,B)$ if $G$ is clear from context.
When it is clear from context, we sometimes identify a hypergraph with its edge set. In particular, we will use $V_1 \times V_2$ for the complete bipartite graph on vertex classes $(V_1,V_2)$, and more generally, $V_1 \times\cdots\times V_k$ for the complete $k$-partite $k$-graph on vertex classes $(V_1,\ldots,V_k)$.




For a partition $\Z$ of a vertex set $V$,
the complete multipartite $k$-graph on $\Z$ is denoted by
$\Cross_k(\Z)=
\big\{ e \sub V \,\big\vert\, \forall\, V \in \Z \colon |e \cap V| \le 1 \text{ and } |e|=k\big\}$.
For partitions $\P,\Q$ of the same underlying set, we say that $\Q$ \emph{refines} $\P$, denoted $\Q \prec \P$, if every member of $\Q$ is contained in a member of $\P$.
We use the notation $x \pm \e$ for a number lying in the interval $[x-\e,\,x+\e]$.

We now define a \emph{$k$-partition},
which is a notion of a hypergraph partition\footnote{This is the standard notion, identical to the one used by R\"odl-Schacht (\cite{RodlSc07}, Definition10).}.
%
%
A $k$-partition $\P$ is of the form $\P=\P^{(1)} \cup\cdots\cup \P^{(k)}$ where $\P^{(1)}$ is a vertex partition, and for each $2 \le r \le k$, $\P^{(r)}$ is a partition of $\Cross_r(\P^{(1)})$ satisfying a condition we will state below.
First, to ease the reader in, let us describe here what a $k$-partition is for $1 \le k \le 3$.
A $1$-partition is simply a vertex partition.
A $2$-partition $\P=\P^{(1)} \cup \P^{(2)}$ consists of a vertex partition $\P^{(1)}$ and a partition $\P^{(2)}$ of $\Cross_2(\P^{(1)})$
such that the complete bipartite graph between any two distinct clusters of $\P^{(1)}$ is a union of parts of $\P^{(2)}$.
A $3$-partition $\P=\P^{(1)} \cup \P^{(2)} \cup \P^{(3)}$ consists of a $2$-partition $\P^{(1)} \cup \P^{(2)}$ and a partition $\P^{(3)}$ of $\Cross_3(\P^{(1)})$
such that for every tripartite graph $G$ whose three vertex clusters lie in $\P^{(1)}$ and three bipartite graphs lie in $\P^{(2)}$, the $3$-partite $3$-graph consisting of all triangles in $G$ is a union of parts of $\P^{(3)}$.


Before defining a $k$-partition in general, we need some terminology.
A \emph{$k$-polyad} is simply a $k$-partite $(k-1)$-graph.
Thus, a $2$-polyad is just a pair of disjoint vertex sets, and a $3$-polyad is a tripartite graph.
In the rest of this paragraph let $P$ be a $k$-polyad on vertex classes $(V_1,\ldots,V_k)$.
We
often identify $P$
with the $k$-tuple $(F_1,\ldots,F_k)$ where each $F_i$ is the induced $(k-1)$-partite $(k-1)$-graph $F_i=P[\bigcup_{j \neq i} V_j]$.
%
We denote by $\K(P)$ the set of $k$-element subsets of $V(P)$ that span a clique (i.e., a $K^{k-1}_{k}$) in $P$; we view $\K(P)$ as a $k$-graph on $V(P)$.
Note that $\K(P)$ is a $k$-partite $k$-graph.
For example, if $P$ is a $2$-polyad then $\K(P)$ is a complete bipartite graph (since $K^1_2$ is just a pair of vertices), and if $P$ is a $3$-polyad then $\K(P)$ is the $3$-partite $3$-graph whose edges correspond to the triangles in $P$.
For a family of hypergraphs $\P$, we say that the $k$-polyad $P=(F_1,\ldots,F_k)$ is a $k$-polyad \emph{of} $\P$ if $F_i \in \P$ for every $1 \le i \le k$.



%


%
We are now ready to define a $k$-partition for arbitrary $k$.

\begin{definition}[$k$-partition]\label{def:r-partition}
	$\P$ is a \emph{$k$-partition} $(k \ge 1)$ on $V$ if $\P=\P^{(1)} \cup\cdots\cup \P^{(k)}$ with $\P^{(1)}$ a partition of $V$,
	and for every $2 \le r \le k$, $\P^{(r)}$ is a partition of $\Cross_r(\P^{(1)})$ into $r$-partite $r$-graphs with $\P^{(r)} \prec \K_r(\P):=\{\K(P) \,\vert\, P \text{ is an $r$-polyad of } \P\}$.
\end{definition}
Note that, by Definition~\ref{def:r-partition}, each $r$-partite $r$-graph $F \in \P^{(r)}$ satisfies $F \sub \K(P)$ for a unique $r$-polyad $P$ of $\P$.
In this context, let $\under$ be the function mapping $F$ to $P$.
So for example, if $\P=\P^{(1)}\cup\P^{(2)}\cup\P^{(3)}$ is a $3$-partition, for every bipartite graph $F \in \P^{(2)}$ we have that $\under(F)$ is the pair of vertex classes of $F$;
similarly, for every $3$-partite $3$-graph $F \in \P^{(3)}$ we have that $\under(F)$ is the unique $3$-partite graph of $\P^{(2)}$ whose set of triangles contains the edges of $F$.



We encourage the reader to verify that Definition~\ref{def:r-partition} is indeed compatible with the explicit description of a $1$-, $2$- and $3$-partition given above.

%
%
%
%

\subsection{$\langle \d \rangle$-regularity of graphs and hypergraphs}

In this subsection we define our new notion of $\langle\d\rangle$-regularity, first for graphs and then for $k$-graphs for any $k \ge 2$ in Definition~\ref{def:k-reg} below. 
\begin{definition}[graph $\langle\d\rangle$-regularity]\label{def:star-regular}
	A bipartite graph $G$ on $(A,B)$ is \emph{$\langle \d \rangle$-regular} if for all subsets $A' \sub A$, $B' \sub B$ with $|A'| \ge \d|A|$, $|B'|\ge\d|B|$ we have $d_G(A',B') \ge \frac12 d_G(A,B)$.\\
	A vertex partition $\P$ of a graph $G$ is \emph{$\langle \d \rangle$-regular}
	if one can add/remove at most $\d \cdot e(G)$ edges so that the bipartite graph induced on each $(A,B)$ with $A \neq B \in \P$ is $\langle \d \rangle$-regular.
	%
\end{definition}

For the reader worried that in Definition~\ref{def:star-regular} we merely replaced the $\e$
from the definition of $\e$-regularity with $\d$, we refer to the discussion following Theorem~\ref{theo:main} below.


%
%

The definition of $\langle\d\rangle$-regularity for hypergraphs involves the $\langle\d\rangle$-regularity notion for graphs, applied to certain auxiliary graphs which are defined as follows.
Henceforth, if $P$ is a $(k-1)$-graph and $H$ is $k$-graph then we say that $H$  is \emph{underlied} by $P$ if $H \sub \K(P)$.



\begin{definition}[The auxiliary graph $G_{H}^i$]\label{def:aux}
	For a $k$-partite $k$-graph $H$ on vertex classes $(V_1,\ldots,V_k)$,
	we define a bipartite graph $G_{H}^1$ on the vertex classes $(V_2 \times\cdots\times V_k,\,V_1)$ by
	$$E(G_{H}^1) = \big\{ ((v_2,\ldots,v_k),v_1) \,\big\vert\, (v_1,\ldots,v_k) \in E(H) \big\} \;.$$
	The graphs $G_{H}^i$ for $2 \le i \le k$ are defined in an analogous manner.
%
%
More generally, if $H$ is underlied by the $k$-polyad $P=(F_1,\ldots,F_k)$ then we define $G_{H,P}^i$ as the induced subgraph $G_{H,P}^i=G_{H}^i[F_i,V_i]$.
\end{definition}



As a trivial example, if $H$ is a bipartite graph then $G_H^1$ and $G_H^2$ are both isomorphic to $H$.

%

Importantly, for a $k$-partition (as defined in Definition~\ref{def:r-partition}) to be $\langle\d\rangle$-regular it must first satisfy a requirement on the regularity of its parts.


\begin{definition}[$\langle\d\rangle$-good partition]\label{def:k-good}
A $k$-partition $\P$ on $V$ is \emph{$\langle\d\rangle$-good} if for every $2 \le r \le k$ and every $F \in \P^{(r)}$ the following holds;
letting $P=\under(F)$ be the $r$-polyad of $\P$ underlying $F$, for every $1 \le i \le r$ the bipartite graph $G_{F,P}^i$ is $\langle \d \rangle$-regular.
\end{definition}
Note that a $1$-partition is trivially $\langle \d \rangle$-good for any $\d$.
Moreover, a $2$-partition $\P$ is $\langle \d \rangle$-good if and only if every bipartite graph in $\P^{(2)}$ (between any two distinct vertex clusters of $\P^{(1)}$) is $\langle \d \rangle$-regular (recall the remark after Definition \ref{def:aux}).




For a $(k-1)$-partition $\P$ with $\P^{(1)} \prec \{V_1,\ldots,V_k\}$ we henceforth denote, for every $1 \le i \le k$,
\begin{equation}\label{eq:partition-notation}
V_i(\P) = \Big\{Z \in \P^{(1)} \,\vert\, Z \sub V_i\Big\} \quad\text{ and }\quad E_i(\P) = \Big\{E \in \P^{(k-1)} \,\vert\, E \sub \prod_{j \neq i} V_j\Big\}.\footnote{$\prod_{j \neq i} V_j = V_1\times\cdots\times V_{i-1}\times V_{i+1}\times\cdots\times V_k$.} 
\end{equation} 

\begin{definition}[$\langle\d\rangle$-regular partition]\label{def:k-reg}
	Let $H$ be a $k$-partite $k$-graph on vertex classes $(V_1,\ldots,V_k)$
	and $\P$ be a $\langle \d \rangle$-good $(k-1)$-partition with $\P^{(1)} \prec \{V_1,\ldots,V_k\}$.
	We say that $\P$ is a \emph{$\langle \d \rangle$}-regular partition of $H$ if
	for every $1 \le i \le k$,
	$E_i(\P) \cup V_i(\P)$ is a $\langle \d \rangle$-regular partition of $G_H^i$.
	%
%
%
	%
	%
\end{definition}
Note that for $k=2$, Definition~\ref{def:k-reg} reduces to Definition~\ref{def:star-regular}.
For $k=3$, a $\langle \d \rangle$-regular partition of a $3$-partite $3$-graph $H$ on vertex classes $(V_1,V_2,V_3)$ is a $2$-partition $\P=\P^{(1)} \cup \P^{(2)}$ satisfying that: $(i)$ $\P$ is $\langle \d \rangle$-good per Definition~\ref{def:k-good};
$(ii)$ from the auxiliary graph
$G_H^1$, on $(V_2 \times V_3,\,V_1)$,
one can add/remove at most $\d$-fraction of the edges such that for every graph $F \in E_1(\P)$ (so $F \sub V_2 \times V_3$) 
and every vertex cluster $V \in V_1(\P)$ (so $V \sub V_1$), 
the induced bipartite graph $G_H^1[F,V]$ is $\langle \d \rangle$-regular; and~$(iii)$ the analogues of~$(ii)$ in $G_H^2$ and $G_H^3$ hold as well.

\subsection{Formal statement of the main result}

We are now ready to formally state our Ackermann-type lower bound for $k$-graph $\langle \d \rangle$-regularity (the formal version of
Theorem~\ref{thm:main-informal} above). Recall that we set $\Ack_{1}(x)=2^x$ and then define for every $k \geq 1$ the $(k+1)^{th}$
Ackermann function $\Ack_{k+1}(n)$ to be the $n$-times composition of $\Ack_{k}(n)$, that is,
$\Ack_{k+1}(n)= \Ack_{k}(\Ack_{k}(\cdots(\Ack_{k}(1))\cdots))$.

\begin{theo}[Main result]\label{theo:main}
	The following holds for every $k \ge 2$, $s \in \N$.
	There exists a $k$-partite $k$-graph $H$ of density $2^{-s-k}$,
	and a partition $\V_0$ of $V(H)$ with $|\V_0| \le 2^{200}k$, such that for every $\langle 2^{-16^k} \rangle$-regular partition $\P$ of $H$, if $\P^{(1)} \prec \V_0$ then $|\P^{(1)}| \ge \Ack_k(s)$.
%
\end{theo}

Let us draw the reader's attention to an important and perhaps surprising aspect of Theorem~\ref{theo:main}.
All the known tower-type lower bounds for graph regularity depend on the error parameter $\epsilon$,
that is, they show the existence of graphs $G$ with the property that every $\epsilon$-regular partition of $G$ is of order at least $\Ack_2(\poly(1/\e))$.
This should be contrasted with the fact that our lower bounds for $\langle \d \rangle$-regularity holds for a {\em fixed} error parameter $\delta$.
Indeed, instead of the dependence on the error parameter, our lower bound depends on the {\em density} of the graph.
This delicate difference makes it possible for us to execute the inductive part of the proof of Theorem~\ref{theo:main}.

In \cite{MoshkovitzSh18} we gave a $\Ack_2(\log 1/p)$ upper bound for a notion of regularity that is 
slightly stronger than $\langle \d \rangle$-regularity, where we use $p$ to denote the edge density of a graph/$k$-graph. 
This allowed us to devise a new proof of Fox's upper bound
for the graph removal lemma. We believe that it should
be possible to match our lower bound stated in Theorem \ref{theo:main} with an $\Ack_k(\log 1/p)$ upper bound that applies even 
for a slightly stronger notion of regularity analogous to the one used in \cite{MoshkovitzSh16}. 
This should allow one to deduce an $\Ack_k(\log 1/\epsilon)$ upper bound
for the $k$-graph removal lemma. The best known bounds for this problem are (at least) $\Ack_k(\poly(1/\epsilon))$.

\subsection{The core construction and proof overview}

The graph construction in Theorem~\ref{theo:core} below is the main technical result we will need in order to prove Theorem~\ref{theo:main}.
This lemma was proved in \cite{MS3}.
We will first need to define ``approximate'' refinement (a notion that goes back to Gowers~\cite{Gowers97}).
\begin{definition}[Approximate refinements]
For sets $S,T$ we write $S \sub_\b T$ if $|S \sm T| \le \b|S|$.
For a partition $\P$ we write $S \in_\b \P$ if $S \sub_\b P$ for some $P \in \P$.
For partitions $\P,\Q$ of the same set of size $n$ we write $\Q \prec_\b \P$ if
$$\sum_{\substack{Q \in \Q\colon\\Q \notin_\b \P}} |Q| \le \b n \;.$$
\end{definition}
Note that for $\Q$ equitable, $\Q \prec_\b \P$ if and only if
all but at most $\b|\Q|$ parts $Q \in \Q$ satisfy $Q \in_\b \P$.
We note that throughout the paper we will only use approximate refinements with $\b < 1/2$, and so if $S \in_\b \P$ then $S \sub_\b P$ for a unique $P \in \P$.
We will later need the following claim. 
\begin{claim}\label{claim:refinement-size}
	If $\Q \prec_{1/2} \P$ and $\P$ is equitable then $|\Q| \ge \frac12|\P|$.
\end{claim}
\begin{proof}
	Since $\Q \prec_{1/2} \P$, the underlying set $U$ has a subset $U^*$ of size $|U^*|\ge \frac12|U|$ such that the partitions $\Q^*=\{Q \cap U^* \colon Q \in \Q\}$ and $\P^*=\{P \cap U^* \colon P \in \P\}$ of $U^*$ satisfy $\Q^* \prec \P^*$.
	Since $\P$ is equitable, $|\P^*| \ge \frac12|\P|$.
	Therefore, $|\Q| \ge |\Q^*| \ge |\P^*| \ge \frac12|\P|$, as desired.
\end{proof}

We stress that in Theorem~\ref{theo:core} below we only use notions related to graphs. In particular, $\langle \d \rangle$-regularity refers to Definition~\ref{def:star-regular}.

\begin{lemma}[\cite{MS3}]\label{theo:core}
	Let $\Lside,\Rside$ be disjoint sets. Let
	$\L_1 \succ \cdots \succ \L_s$ and $\R_1 \succ \cdots \succ \R_s$ be two sequences of $s$ successively refined equipartitions of $\Lside$ and $\Rside$, respectively,
	that satisfy for  every $i \ge 1$ that:
	\begin{enumerate}
		\item\label{item:core-minR}
		$|\R_i|$ is a power of $2$ and $|\R_1| \ge 2^{200}$,
		\item\label{item:core-expR} $|\R_{i+1}| \ge 4|\R_i|$,
		\item\label{item:core-expL} $|\L_i| = 2^{|\R_i|/2^{i+10}}$.
	\end{enumerate}
	%
	%
	Then there exists a sequence of $s$ successively refined edge equipartitions $\G_1 \succ \cdots \succ \G_s$ of $\Lside \times \Rside$ such that for every $1 \le j \le s$, $|\G_j|=2^j$, and
	the following holds for every $G \in \G_j$ and $\d \le 2^{-20}$.
	For every $\langle \d \rangle$-regular partition $\P \cup \Q$ of $G$, where $\P$, $\Q$ are partitions of $\Lside$, $\Rside$, respectively, and every $1 \le i \le j$, 	
	if $\Q \prec_{2^{-9}} \R_{i}$ then $\P \prec_{\g} \L_{i}$ with
	$\g = \max\{2^{8}\sqrt{\d},\, 32/\sqrt[6]{|\R_1|} \}$.
\end{lemma}
\begin{remark}
	Every $G \in \G_j$ is a bipartite graph of density $2^{-j}$ since $\G_j$ is equitable.
\end{remark}

Let us end this section by explaining the role Theorem~\ref{theo:core} plays in the proof of Theorem~\ref{theo:main}.


\paragraph{Using graphs to construct $k$-graphs:}
It is not hard to use Theorem \ref{theo:core} in order to prove a tower-type lower bound
for graph $\langle \d \rangle$-regularity. Perhaps the most surprising aspect of the proof
of Theorem~\ref{theo:main} is that in order to construct a $k$-graph we also use the graph construction
of Theorem~\ref{theo:core} in a somewhat unexpected way.
In this case, $\Lside$ will be a complete $(k-1)$-partite $(k-1)$-graph and the $\L_i$'s will be partitions of this complete $(k-1)$-graph themselves given by another application of Theorem~\ref{theo:core}.
The size of the partitions will be of $\Ack_{k-1}$-type growth, and this application of Theorem~\ref{theo:core} will ``multiply'' the $(k-1)$-graph partitions (given by the $\L_i$'s) to
produce a partition of the complete $k$-partite $k$-graph into $k$-graphs that are hard for $\langle \d \rangle$-regularity.
We will take $H$ in Theorem~\ref{theo:main} to be an arbitrary $k$-graph in this partition.


\paragraph{Why is Theorem \ref{theo:core} one-sided?}
As is evident from the statement of Theorem~\ref{theo:core}, it is one-sided in nature; that is, under the premise that the partition $\Q$
refines $\R_i$ we may conclude that $\P$ refines $\L_i$.
It is natural to ask if one can do away with this assumption, that is, be able to show that,
under the same assumptions, $\Q$ refines $\R_i$ and $\P$ refines $\L_i$.
As we mentioned in the previous item, in order to prove an Ackermann-type lower bound for hypergraph
regularity we have to apply Theorem~\ref{theo:core} with a sequence of partitions whose size grows as an Ackermann function of the same level.
Now, in this setting, Theorem~\ref{theo:core} does not hold without the one-sided assumption, because if it did, then
one would have been able to prove an Ackermann-type lower bound for graph $\langle \d \rangle$-regularity, and hence also for Szemer\'edi's regularity lemma.
Put differently, if one wishes to have a construction that holds with arbitrarily fast growing partition sizes, then
one has to introduce the one-sided assumption.

\paragraph{How do we remove the one-sided assumption?}
The proof of Theorem \ref{theo:main} proceeds by first proving a one-sided version of Theorem \ref{theo:main},
stated as Lemma~\ref{lemma:ind-k}. In order to get a construction that does not require such a one-sided assumption,
we will need one final trick; we will take $2k$ clusters of
vertices and arrange $2k$ copies
of this one-sided construction
along the $k$-edges of a cycle. This will give us a ``circle of implications'' that will eliminate the one-sided assumption.
See Subsection~\ref{subsec:pasting}.

\section{Proof of Theorem~\ref{theo:main}}\label{sec:LB}

\renewcommand{\k}{r}

\newcommand{\w}{w}
\renewcommand{\t}{t}

\newcommand{\GG}{\mathbf{G}}
\newcommand{\FF}{\mathbf{F}}
\newcommand{\VV}{\mathbf{V}}

\renewcommand{\Hy}[1]{H_{{#1}}}
\renewcommand{\A}{A}

\newcommand{\subs}{\subset_*}
\newcommand{\pad}{P}
\renewcommand{\K}{\mathcal{K}}
\newcommand{\U}{U}

\renewcommand{\k}{k}

\renewcommand{\K}{\mathcal{K}}
\renewcommand{\r}{k}

The purpose of this section is to prove the main result, Theorem~\ref{theo:main}. 
Its proof crucially relies on a subtle inductive argument (see Lemma~\ref{lemma:ind-k} below).
This section is self-contained save for the application of Theorem~\ref{theo:core}.
The key step of our lower bound proof for $k$-graph regularity, stated as Lemma~\ref{lemma:ind-k} and proved in Subsection~\ref{subsec:main-induction}, relies on a construction that applies Theorem~\ref{theo:core} $k-1$ times. This lemma only gives a ``one-sided'' lower bound, in the spirit of Theorem~\ref{theo:core}.
In Subsection~\ref{subsec:pasting} we show how to use Lemma~\ref{lemma:ind-k} in order to complete the proof of Theorem~\ref{theo:main}.

We first state some properties of $k$-partitions whose proofs are deferred to the end of this section.
The first property relates $\d$-refinements of partitions and $\langle \d \rangle$-regularity of partitions. 
The reader is advised to recall the notation in~(\ref{eq:partition-notation}).


\begin{claim}\label{claim:uniform-refinement}
	Let $\P$ be a $(k-1)$-partition with  $\P^{(1)} \prec \{V_1,\ldots,V_k\}$,
	and let $\F$ be a partition of $V_1\times\cdots\times V_{k-1}$ with $E_k(\P) \prec_\d \F$.
	If $\P$ is $\langle \d \rangle$-good
	then the $(k-2)$-partition $\P'$ obtained by restricting $\P$ to $\bigcup_{i=1}^{k-1} V_i$ is a $\langle 3\d \rangle$-regular partition of some $F \in \F$.
\end{claim}

The second property is given by following easy (but slightly tedious to state) claim.
\begin{claim}\label{claim:restriction}
	Let $H$ be a $k$-partite $k$-graph on vertex classes $(V_1,\ldots,V_k)$, and let $H'$ be the induced $k$-partite $k$-graph on vertex classes $(V_1',\ldots,V_k')$ with $V_i' \sub V_i$ and $\b \cdot e(H)$ edges.
	If $\P$ is a $\langle \d \rangle$-regular partition of $H$ with
	$\P^{(1)} \prec \bigcup_{i=1}^k \{V_i,\,V_i \sm V_i'\}$
	then its restriction $\P'$ to $V(H')$ is a $\langle \d/\b \rangle$-regular partition of $H'$.
\end{claim}

\renewcommand{\K}{k}

\subsection{Key inductive proof}\label{subsec:main-induction}

\paragraph*{Set-up.}

We next introduce a few more definitions that are needed for the statement of Lemma~\ref{lemma:ind-k}.
Denote $e(i) = 2^{i+10}$. We define the following tower-type function $\t\colon\N\to\N$;
\begin{equation}\label{eq:t}
\t(i+1) = \begin{cases}
2^{\t(i)/e(i)}	&\text{if } i \ge 1\\
2^{200}	&\text{if } i = 0
\end{cases}
\end{equation}
It is easy to prove, by induction on $i$, that $\t(i) \ge e(i)\t(i-1)$ for $i \ge 2$ (for the induction step,
$t(i+1) \ge 2^{\t(i-1)} = t(i)^{e(i-1)}$, so $t(i+1)/e(i+1) \ge \t(i)^{e(i-1)-i-11} \ge \t(i)$).
This means that $t$ is a monotone increasing function, and that it is always an integer power of $2$ (follows by induction as $t(i)/e(i) \ge 1$ is a positive power of $2$ and in particular an integer).
We record the following facts regarding $\t$ for later use:
\begin{equation}\label{eq:monotone}
\t(i) \ge 4\t(i-1)  \quad\text{ and }\quad
\text{ $\t(i)$ is a power of $2$} \;.
\end{equation}
For a function $f:\N\to\N$ with $f(i) \ge i$ we denote
\begin{equation}\label{eq:f*}
f^*(i) = \t\big(f(i)\big)/e(i) \;.
\end{equation}
Note that $f^*(i)$ is indeed a positive integer (by the monotonicity of $\t$, $f^*(i) \ge \t(i)/e(i)$ is a positive power of $2$).
In fact, $f^*(i) \ge f(i)$ (as $f^*(i) \ge 4^{f(i)}/e(i)$ using~(\ref{eq:monotone})).

We recursively define the function $\A_\k\colon\N\to\N$ for any integer $\k \ge 2$ as follows: $\A_2(i)=i$, whereas for $\k \ge 3$, 
\begin{equation}\label{eq:Ak}
\A_\k(i+1) = \begin{cases}
\A_{\k-1}(\A_\k^*(i))		&\text{if } i \ge 1\\
2^{2^{3\K+2}}	&\text{if } i = 0
\end{cases}
\end{equation}
Note that $\A_\k$ is well defined since $\A_\k^*(i) \in \N$ for $i \ge 1$ and $k \ge 2$, and that $\A_\k^{},\A_\k^*$ are both monotone increasing.
It is evident that $\A_\k$ grows like the $k$-th level Ackermann function; in fact, one can check that for every $\k \ge 2$ we have
\begin{equation}\label{eq:A_k}
\A_\k(i) \ge \Ack_\k(i) \;.
\end{equation}
Furthermore, we denote, for $\k \ge 1$,
\begin{equation}\label{eq:delta_k}
\d_\k = 2^{-8^\k} \;.
\end{equation}
We moreover denote, for $\k \ge 2$,
\begin{equation}\label{eq:m_k}
m_\k(i):=\A_2^*(\cdots(\A^*_{\k}(i))\cdots) \;.
\end{equation}
%
We next record a few easy bounds for later use.
Recall that
\begin{equation}\label{eq:delta_12}
\d_1 = 2^{-8} \quad\text{ and }\quad \d_2 = 2^{-64} \;.
\end{equation}
Noting the relation $\d_{\k} = \d_{\k-1}^8$, we have for $\k \ge 3$ that
\begin{equation}\label{eq:delta_k-bound}
\d_\k^{1/4} = \d_{\k-1}^2 \le \d_2 \d_{\k-1} = 2^{-64}\d_{\k-1} \;.
\end{equation}
Noting the relation $\A_\K(1) = \d_\K^{-4}$ for $\K \ge 3$, we have for $\K \ge 3$ that
\begin{equation}\label{eq:t1-bound}
1/\sqrt[6]{\A_\K(1)} \le \d_{\K}^{1/2} = \d_{\K-1}^{4} \le \d_1^3\d_{\K-1} = 2^{-24}\d_{\K-1} \;.
\end{equation}


\paragraph*{Key inductive proof.}

The key argument in our lower bound proof for $\K$-graph regularity is the following theorem, which is proved by induction on the hypergraph's uniformity.






\begin{lemma}[$\k$-graph induction]\label{lemma:ind-k}
	Let $s \in \N$, let $\Vside^1,\ldots,\Vside^\k$ be $\k \ge 2$ mutually disjoint sets of equal size $n$ and let $\V_1 \succ\cdots\succ \V_m$ be a sequence of
	$m=m_\k(s)$ 
	successive equitable refinements of $\{\Vside^1,\ldots,\Vside^\k\}$
	with $|V_h(\V_i)|=\t(i)$ for every $i,h$.\footnote{Since we assume that each $\V_i$ refines $\{\Vside^1,\ldots,\Vside^k\}$ then $V_1(\V_i)$ is (by the notation mentioned before Claim \ref{claim:star-union}) the restriction of $\V_i$ to $\Vside^1$.}
	Then there exists a sequence of $s$ successively refined equipartitions $\HH_1 \succ \cdots \succ \HH_s$ of $\Vside^1 \times \cdots \times \Vside^\k$ such that for every $1 \le j \le s$, $|\HH_j|=2^j$
	and every $H \in \HH_j$ satisfies the following property:\\
		If $\P$ is a $\langle \d_\k \rangle$-regular
		partition of $H$, and for some $1 \le i \le \A_\k(j)$ we have $V_h(\P) \prec_{2^{-9}} V_h(\V_i)$ for every $2 \le h \le \k$, then $V_1(\P) \prec_{2^{-9}} V_1(\V_{i+1})$.
\end{lemma}
Note that $\V_i$ is well defined in the property described in Lemma~\ref{lemma:ind-k} since $i \le \A_\k(j) \le m$.
\begin{proof}	
	We proceed by induction on $\k \ge 2$.
	For the induction basis $\k=2$ we are given $s \in \N$, two disjoint sets $\Vside^1,\Vside^2$ 
	as well as $m=\A_2^*(s)$ ($\ge s+1$) successive equitable refinements $\V_1 \succ\cdots\succ \V_{m}$ of $\{\Vside^1,\Vside^2\}$.
	Our goal is to find a sequence of $s$ successively refined equipartitions $\HH_1 \succ \cdots \succ \HH_s$ of $\Vside^1 \times \Vside^2$ as in the statement.
	To prove the induction basis, apply Theorem~\ref{theo:core} with 
	$$\Lside=\Vside^1,\quad \Rside=\Vside^2 \quad\text{ and }\quad V_1(\V_2) \succ \cdots \succ V_1(\V_{s+1}) ,\quad V_2(\V_1) \succ \cdots \succ V_2(\V_{s}) \;,$$
	and let
	$$\G^1 \succ \cdots \succ \G^{s} \quad\text{ with }\quad |G^\ell|=2^\ell \text{ for every } 1 \le \ell \le s $$
	be the resulting sequence of $s$ successively refined equipartitions of $\Vside^1 \times \Vside^2$.
	These two sequences indeed satisfy assumptions~\ref{item:core-minR},~\ref{item:core-expR} in Theorem~\ref{theo:core} since $|V_2(\V_j)|=\t(j)$ and by~(\ref{eq:monotone}); 	
	moreover, they satisfy assumption~\ref{item:core-expL} since for every $1 \le j \le s$ we have
	$$|V_1(\V_{j+1})| = \t(j+1) = 2^{\t(j)/e(j)} = 2^{|V_2(\V_{j})|/e(j)} \;,$$
	where the second equality uses the definition of the function $\t$ in~(\ref{eq:t}).
	We will show that taking $\HH_j=\G_j$ for every $1 \le j \le s$ yields a sequence as required by the statement.
	Fix $1 \le j \le s$ and $G \in \G_j$; note that $G$ is a bipartite graph on the vertex classes $(\Vside^1,\Vside^2)$.
	Moreover, let $1 \le i \le j$ (recall $\A_2(j)=j$) and let $\P$ be a $\langle \d_2 \rangle$-regular 
	partition of $G$
	with $V_2(\P) \prec_{2^{-9}} V_2(\V_{i})$.
	Since $\d_2 \le 2^{-20}$ by~(\ref{eq:delta_12}),
	Theorem~\ref{theo:core} implies that $V_1(\P) \prec_{x} V_1(\V_{i+1})$ with $x=\max\{2^{8}\sqrt{\d_2},\, 32/\sqrt[6]{\t(1)} \}$.
	Using~(\ref{eq:delta_12}) and~(\ref{eq:t})
	we have $x \le 2^{-9}$, completing the proof of the induction basis.
	
	To prove the induction step, recall that we are given $s \in \N$, disjoint sets $\Vside^1,\ldots,\Vside^\k$ 
	and a sequence of $m=m_\k(s)$ successive equitable refinements $\V_1 \succ\cdots\succ \V_m$ of $\{\Vside^1,\ldots,\Vside^\k\}$,
	and our goal is to construct a sequence of $s$ successively refined equipartitions $\HH_1 \succ \cdots \succ \HH_s$ of $\Vside^1 \times \cdots \times \Vside^\k$ as in the statement.
	We begin by applying the induction hypothesis with $\k-1$ (which would imply Proposition~\ref{prop:main-k-hypo} below).
	We henceforth put $c=2^{-9}$.
	Now, apply the induction hypothesis with $\k-1$ on
	\begin{equation}\label{eq:ind-hypo}
	\text{
		$s':=\A_\k^*(s)$ (in place of $s$),\, $\Vside^1,\ldots,\Vside^{\k-1}$ and $\bigcup_{h=1}^{\k-1}V_h(\V_1) \succ\cdots\succ \bigcup_{h=1}^{\k-1}V_h(\V_m)$} \;,
	\end{equation}
	and let
	\begin{equation}\label{eq:main-k-colors}
	\F_1 \succ \cdots \succ \F_{s'} \quad\text{ with }\quad |F_ \ell|=2^\ell \text{ for every } 1 \le \ell \le s'
	\end{equation}
	be the resulting sequence of $s'$ successively refined equipartitions of $\Vside^1 \times \cdots \times \Vside^{\k-1}$.
	\begin{prop}[Induction hypothesis]\label{prop:main-k-hypo}
		Let $1 \le \ell \le s'$ and $F \in \F_\ell$.
		If $\P'$ is a $\langle \d_{\k-1} \rangle$-regular partition of $F$ with ${\P'}^{(1)} \prec \{\Vside^1,\ldots,\Vside^{\k-1}\}$, and for some $1 \le i \le \A_{\k-1}(\ell)$ we have
		$V_h(\P') \prec_c V_h(\V_i)$ for every $2 \le h \le \k$,
		then $V_1(\P) \prec_c V_1(\V_{i+1})$.
	\end{prop}
	\begin{proof}
		It suffices to verify that the number $m$ of partitions in~(\ref{eq:ind-hypo}) is as required by the induction hypothesis. Indeed, by~(\ref{eq:m_k}),
		$$m_{\k-1}(s') = \A_2^*(\cdots(\A_{\k-1}^*(s'))\cdots) = \A_2^*(\cdots(\A^*_{\k}(s))\cdots) = m_\k(s) = m \;.$$	
	\end{proof}
For each $1 \le j \le s$ let
\begin{equation}\label{eq:main-k-dfns}
\F_{(j)} = \F_{\A_k^*(j)}
\quad\text{ and }\quad
\V_{(j)} = V_k(\V_{\A_k(j)}) \;.
\end{equation}	
All these choices are well defined since $\A_\k^*(j)$ satisfies $1 \le \A_\k^*(1) \le \A_\k^*(j) \le \A_\k^*(s) = s'$ by our choice of $s'$ in~(\ref{eq:ind-hypo}), and since $\A_\k(j)$ satisfies $1 \le \A_\k(1) \le \A_\k(j) \le \A_\k(s) \le m$. 
Observe that we have thus chosen two subsequences of $\F_1,\cdots,\F_{s'}$ and $V_k(\V_1),\ldots,V_k(\V_m)$, each of length $s$.
Recalling that each $\F_{(j)}$ is a partition of $\Vside^1 \times\cdots\times \Vside^{k-1}$, apply Theorem~\ref{theo:core} with
$$
\Lside=\Vside^1 \times\cdots\times \Vside^{k-1},\quad \Rside=\Vside^k \quad\text{ and }\quad \F_{(1)} \succ \cdots \succ \F_{(s)}, \quad \V_{(1)} \succ \cdots \succ \V_{(s)} \;,
$$
and let
\begin{equation}\label{eq:ind-colors2}
\G_1 \succ \cdots \succ \G_{s} \quad\text{ with }\quad |G_\ell|=2^\ell \text{ for every } 1 \le \ell \le s
\end{equation}
be the resulting sequence of $s$ successively refined graph equipartitions of $(\Vside^1\times\cdots\times\Vside^{\k-1})\times\Vside^\k$.

\begin{prop}[Core proposition]\label{prop:ind-prop2}
	Let $1 \le j \le s$ and $G \in \G_j$.
	If $\E \cup \V$ is a $\langle \d_\k \rangle$-regular partition of $G$ (where $\E$ and $\V$ are partitions of $\Vside^1\times\cdots\times\Vside^{\k-1}$ and $\Vside^\k$ respectively),
	and for some $1 \le j' \le j$ we have $\V \prec_{c} \V_{(j')}$, 
	then $\E \prec_{\frac14\d_{\k-1}} \F_{(j')}$.
\end{prop}
\begin{proof}
	First we need to verify that we may apply Theorem~\ref{theo:core} as above.
	Assumptions~\ref{item:core-minR} and~\ref{item:core-expR} follow from the fact that $|\V_{(j)}|=\t(\A_\k(j))$, from~(\ref{eq:monotone})
	and the fact that $\A_k(1) \ge 2^{2^{11}} \ge 2^{200}$ for $\k \ge 3$ by~(\ref{eq:Ak}).	
	To verify that assumption~\ref{item:core-expL} holds, note that $|\F_{(j)}|=2^{\A_\k^*(j)}$
	by~(\ref{eq:main-k-colors}),
	and that
	$|\V_{(j)}|=\t(\A_\k(j))$
	by the statement's assumption that $|\V_i[\k]|=\t(i)$.
	Thus, indeed, 
	$$
	|\F_{(j)}| = 2^{\A_\k^*(j)} = 2^{\t(\A_\k(j))/e(j)} = 2^{|\V_{(j)}|/e(j)} \;,
	$$
	where the second equality uses the definition in~(\ref{eq:f*}).
	%
	Moreover, note that $\d_\k \le \d_2 \le 2^{-20}$ by~(\ref{eq:delta_12}).
	We can thus use Theorem~\ref{theo:core} to infer that the fact that $\V \prec_{c} \V_{(j')}$ implies that $\E \prec_x \F_{(j')}$ with $x=\max\{2^{5}\sqrt{\d_\k},\, 32/\sqrt[6]{\t(\A_\K(1))} \}$.
	To see that indeed $x \le \frac14\d_{\k-1}$,
	apply~(\ref{eq:delta_k-bound}) as well as the fact that $\t(\A_\K(1)) \ge \A_\K(1)$ and~(\ref{eq:t1-bound}).
\end{proof}

	For each $G \in \G_j$ let $\Hy{G}$ be the $k$-partite $k$-graph on vertex classes $(\Vside^1,\ldots,\Vside^k)$ with edge set
	$$E(\Hy{G}) = \big\{ (v_1,\ldots,v_k) \,:\, ((v_1,\ldots,v_{k-1}),v_k) \in E(G) \big\} \;,$$
	and note that we have (recall Definition \ref{def:aux})
	\begin{equation}\label{eqH}
	G=G_{\Hy{G}}^k\;.
	\end{equation}
	For every $1 \le j \le s$ let $\HH_j=\{\Hy{G} \colon G \in \G_j \}$, and note that $|\HH_j|=|\G_j|=2^j$ by~(\ref{eq:ind-colors2}),
	that $H_j$ is an equipartition of $\Vside^1\times\cdots\times\Vside^k$,
	and that $\HH_1 \succ\cdots\succ \HH_s$.
	Our goal is to show that these partitions satisfy the property guaranteed by the statement.
	
	Henceforth fix $1 \le j \le s$ and $H \in \HH_j$, and write $H=\Hy{G}$ with $G \in \G_j$.
	To complete the proof is suffices to show that $H$ satisfies the property is the statement.
	Assume now that $i$ is such that
	\begin{equation}\label{eq:ind-i-assumption}
	1 \le i \le \A_\k(j)
	\end{equation}
	and:
	\begin{enumerate}
		\item\label{item:ind-reg}
		$\P$ is a $\langle \d_\k \rangle$-regular partition of $H$,
		\item\label{item:ind-refine} $V_h(\P) \prec_{c} V_h(\V_i)$ for every $2 \le h \le \k$.
	\end{enumerate}	
	In the remainder of the proof we will complete the induction step by showing that
	\begin{equation}\label{eq:ind-goal}
	V_1(\P) \prec_{c} V_1(\V_{i+1}) \;.
	\end{equation}
	It follows from Item~\ref{item:ind-reg},
	by Definition~\ref{def:k-reg} and~(\ref{eqH}), that in particular
	\begin{equation}\label{eq:ind-reg}
	E_k(\P) \cup V_k(\P) \text{ is a $\langle \d_\k \rangle$-regular partition of } G.
	\end{equation}
		Let
	\begin{equation}\label{eq:ind-j'}
	1 \le j' \le s
	\end{equation}
	be the unique integer satisfying 
	\begin{equation}\label{eq:ind-sandwich}
	\A_k(j') \le i < \A_k(j'+1) \;.
	\end{equation}
	Note that (\ref{eq:ind-j'}) holds due to~(\ref{eq:ind-i-assumption}).
	Recalling~(\ref{eq:main-k-dfns}),
	the lower bound in~(\ref{eq:ind-sandwich}) implies that $V_k(\V^i) \prec V_k(\V_{\A_k(j')}) = \V_{(j')}$.
	Therefore, the assumption $V_k(\P) \prec_{c} V_k(\V^i)$ in Item~\ref{item:ind-refine} implies that
	\begin{equation}\label{eq:ind-Zk}
	V_k(\P) \prec_{c} \V_{(j')} \;.
	\end{equation}
	Apply Proposition~\ref{prop:ind-prop2} on $G$, using~(\ref{eq:ind-reg}),~(\ref{eq:ind-j'}) and~(\ref{eq:ind-Zk}), to deduce that
	\begin{equation}\label{eq:ind-E}
	E_k(\P) \prec_{\frac14\d_{k-1}} \F_{(j')} = \F_{\A_k^*(j')} \;,
	\end{equation}
	where for the equality again recall~(\ref{eq:main-k-dfns}).
	Let $\P^*$ be the restriction of $\P$ to $\Vside^1\cup\cdots\cup\Vside^{\k-1}$,
	and let $\P' = \P^* \sm \P^*[\Vside^1\times\cdots\times\Vside^{\k-1}]$.
	Note that $\P^*$ is a $(\k-1)$-partition on $(\Vside^1,\ldots,\Vside^{\k-1})$ and that $\P'$ is a $(\k-2)$-partition on $(\Vside^1,\ldots,\Vside^{\k-1})$.
	Since $\P$ is a $\langle \d_k \rangle$-regular partition of $H$ (by Item~\ref{item:ind-reg} above), $\P^*$
	is in particular $\langle \d_k \rangle$-good. By~(\ref{eq:ind-E}) we may thus apply Claim~\ref{claim:uniform-refinement} on $\P^*$
	to conclude that
	\begin{equation}\label{eq:ind-reg2}
	\P' \text{ is a }
	\langle \d_{k-1} \rangle
	 \text{-regular partition of some $F\in\F_{\A_k^*(j')}$.}
	\end{equation}
	By~(\ref{eq:ind-reg2}) we may apply Proposition~\ref{prop:main-k-hypo} with $G$, $\P'$, $\ell=\A_k^*(j')$ and $i$, observing (crucially)
	that $i \leq \ell$ by (\ref{eq:ind-sandwich}). 
	Note that Item~\ref{item:ind-refine} in particular implies that $V_h(\P') \prec_{c} V_h(\V_i)$ for every $2 \le h \le k$.
	We thus deduce that $V_1(\P') \prec_{c} V_1(\V_{i+1})$.
	Since $V_1(\P') = V_1(\P)$, this proves~(\ref{eq:ind-goal}) and thus completes the induction step and the proof.
\end{proof}

\subsection{Putting everything together}\label{subsec:pasting}




We can now prove our main theorem, Theorem~\ref{theo:main}, which we repeat here for convenience.
\addtocounter{theo}{-1}
\begin{theo}[Main theorem]\label{theo:main}
	The following holds for every $k \ge 2$, $s \in \N$.
	There exists a $k$-partite $k$-graph $H$ of density $2^{-s-k}$,
	and a partition $\V_0$ of $V(H)$ with $|\V_0| \le 2^{200}k$, such that if $\P$ is a $\langle 2^{-16^k} \rangle$-regular partition of $H$ with $\P^{(1)} \prec \V_0$ then $|\P^{(1)}| \ge \Ack_k(s)$.
%
%
	
%
%
\end{theo}

\begin{remark}\label{remark:main}
	As can be easily checked, the proof of Theorem~\ref{theo:main} also gives that $H$ has the same number of vertices in all vertex classes.
\end{remark}


\begin{proof}
	Let the $k$-graph $B$ be the tight $2k$-cycle; that is, $B$ is the $k$-graph on vertex classes $\{0,1,\ldots,2k-1\}$ with edge set $E(B)=\{\{0,1,\ldots,k-1\},\{1,2,\ldots,k\},\ldots,\{2k-1,0,\ldots,k-2\}\}$.
	Note that $B$ is $k$-partite with vertex classes $(\{0,k\},\{1,k+1\},\ldots,\{k-1,2k-1\}\}$.
	Put $m=m_k(s-k)$ and let $n \ge \t(m)$.
	Let $\Vside^0,\ldots,\Vside^{2k-1}$ be $2k$ mutually disjoint sets of size $n$ each.
	Let $\V^1 \succ\cdots\succ \V^m$ be an arbitrary sequence of $m$ successive equitable refinements of $\{\Vside^0,\ldots,\Vside^{2k-1}\}$ with $|\V^i_h|=\t(i)$ for every $1 \le i \le m$ and $0 \le h \le 2k-1$, which exists as $n$ is large enough.
	Extending the notation $\V_x$ (above Definition~\ref{def:k-reg}), for every $0 \le x \le 2k-1$ we henceforth denote the restriction of the vertex partition $\V \prec \{\Vside^0,\ldots,\Vside^{2k-1}\}$ to $\Vside^x$ by $\V_x = \{V \in \V \,\vert\, V \sub \Vside^x\}$.
	For each edge $e=\{x,x+1,\ldots,x+k-1\} \in E(B)$ (here and henceforth when specifying an edge, the integers are implicitly taken modulo $2k$)
	apply Lemma~\ref{lemma:ind-k} with
	$$s,\,
	\Vside^{x},\Vside^{x+1},\ldots,\Vside^{x+k-1}
	\text{ and }
	\bigcup_{j=0}^{k-1}\V^{1}_{x+j}
	\succ\cdots\succ \bigcup_{j=0}^{k-1}\V^{m}_{x+j} \;.$$
	Let $H_e$ denote the resulting $k$-partite $k$-graph on $(\Vside^{x},\Vside^{x+1},\ldots,\Vside^{x+k-1})$.
	Note that $d(H_e) = 2^{-s}$.
	Moreover, denoting
	$$c = 2^{-9} \quad\text{ and }\quad K=\A_k(s)+1 \;,$$
	$H_e$ has the property that for every $\langle \d_k \rangle$-regular partition $\P'$ of $H_e$ and every $1 \le i < K$,
	\begin{equation}\label{eq:paste-property}
	\text{If $V_{x+h}(\P') \prec_{c} V_{x+h}(\V_i)$ for every $1 \le h \le k-1$, then $V_x(\P) \prec_{c} V_x(\V_{i+1})$.}
	\end{equation}	
	We construct our $3$-graph on the vertex set $\Vside:=\Vside^0 \cup\cdots\cup \Vside^{2k-1}$ as
	$E(H) = \bigcup_{e} E(H_e)$; that is, $H$ is the edge-disjoint union of all $2k$ $k$-partite $k$-graphs $H_e$ constructed above.
	Note that $H$ is a $k$-partite $k$-graph (on vertex classes $(\Vside^0 \cup \Vside^k,\, \Vside^1 \cup \Vside^{k+1},\ldots, \Vside^{k-1} \cup \Vside^{2k-1}))$ of density $\frac{2k}{2^k} 2^{-s} \ge 2^{-s-k}$, as needed.
	We will later use the following fact.
	\begin{prop}\label{prop:restriction}
		Let $\P$ be an $\langle 2^{-16^k} \rangle$-regular partition of $H$ with $\P^{(1)} \prec \{\Vside^0,\ldots,\Vside^{2k-1}\}$, and let
		$e \in E(B)$.
		Then the restriction $\P'$ of $\P$ to $V(H_e)$ is a $\langle \d_k \rangle$-regular partition of $H_e$.
	\end{prop}
	\begin{proof}
		Immediate from Claim~\ref{claim:restriction} using the fact that $e(H_e) = \frac{1}{2k}e(H)$.
		%
	\end{proof}	
	
	Now, let $\P$ be an $\langle 2^{-16^k} \rangle$-regular partition of $H$
	with $\P^{(1)} \prec \V^1$.
	Our goal will be to show that
	\begin{equation}\label{eq:paste-goal}
	\P^{(1)} \prec_{c} \V^{K} \;.
	\end{equation}
	Proving~(\ref{eq:paste-goal}) would complete the proof, by setting $\V_0$ in the statement to be $\V^1$ here (notice $|\V^1|=k\t(1) = k2^{200}$ by~(\ref{eq:t})); 
	indeed, Claim~\ref{claim:refinement-size} would imply that
	$$|\P^{(1)}| \ge \frac12|\V^{K}| = \frac12 \cdot 2k \cdot \t(K)
	\ge \t(K)
	\ge \t(\A_k(s))
	\ge \A_k(s)
	\ge \Ack(s) \;,$$
	where the last inequality uses~$(\ref{eq:A_k})$.
	Assume towards contradiction that $\P^{(1)} \nprec_{c} \V^{K}$. By averaging,
	\begin{equation}\label{eq:assumption}
	\P^{(1)}_h \nprec_c \V^{K}_h \text{ for some } 0 \le h \le 2k-1.
	\end{equation}
	For each $0 \le h \le 2k-1$ let $1 \le \b(h) \le K$ be the largest integer satisfying $\P^{(1)}_h \prec_c \V^{\b(h)}_h$,
	which is well defined since $\P^{(1)}_h \prec_c \V^1_h$ (in fact $\P^{(1)} \prec \V^1$).
	Put $\b^* = \min_{0 \le h \le 2k-1} \b(h)$, and note that by~(\ref{eq:assumption}),
	\begin{equation}\label{eq:paste-star}
	\b^* < K \;.
	\end{equation}
	Let $0 \le x \le 2k-1$ minimize $\b$, that is, $\b(x)=\b^*$.
	Therefore:
	\begin{equation}\label{eqcontra}
	\P^{(1)}_{x+k-1} \prec_c \V^{\b^*}_{x+k-1},\,\ldots,\,\P^{(1)}_{x+1} \prec_c \V^{\b^*}_{x+1} \text{ and }
	\P^{(1)}_{x} \nprec_c \V^{\b^*+1}_{x}.	
	\end{equation}
	Let $e=\{x,x+1,\ldots,x+k-1\} \in E(B)$.
	Let $\P'$ be the restriction of $\P$ to $V(H_e)=\Vside^{x} \cup \Vside^{x+1} \cup\cdots\cup \Vside^{x+k-1}$, which is a $\langle \d_k \rangle$-regular partition of $H_e$ by Proposition~\ref{prop:restriction}. Since ${\P'}^{(x+h)}_{h}=\P^{(x+h)}_h$ for every $0 \le h \le k-1$ we get
	from~(\ref{eqcontra}) a contradiction to~(\ref{eq:paste-property}) with $i=\beta^*$.
	We have thus proved~(\ref{eq:paste-goal}) and so the proof is complete.
\end{proof}

\subsection{Deferred proofs: properties of $k$-partitions}\label{subsec:k-partitions-proofs}

\renewcommand{\K}{\mathcal{K}}

Henceforth, for a $(k-1)$-partite $(k-1)$-graph on $(V_1,\ldots,V_{k-1})$ and a disjoint vertex set $V$ we denote by $F \circ V$ the $k$-partite $k$-graph on $(V_1,\ldots,V_{k-1},V)$ given by
$$F \circ V := \{ (v_1,\ldots,v_{k}) \,\vert\, (v_1,\ldots,v_{k-1}) \in F \text{ and } v_{k} \in V \} \;.$$
%
We will use the following additional property of $k$-partitions.

\begin{claim}\label{claim:decomposition}
	Let $\P$ be a $(k-1)$-partition 
	with $\P^{(1)} \prec (\Vside^1,\ldots,\Vside^k)$,
	$F \in E_k(\P)$ and $V \in V_k(\P)$.
	Then there is a set of $k$-polyads $\{P_i\}_i$ of $\P$ such that
	\begin{equation}\label{eq:red-k-partitionP-gen}
	F \circ V = \bigcup_i \K(P_i) \,\text{ is a partition,
		with } P_i = (P_{i,1},\ldots,P_{i,k-1},F) \;.
	\end{equation}
\end{claim}

\begin{proof}
	We proceed by induction on $\k \ge 2$, noting that the induction basis $\k=2$ is trivial since in this case
	$F=V' \in V_2(\P)$,
	hence $F \circ V = V' \times V$ is simply $\K(P)$ where $P$ is the $2$-polyad of $\P$ corresponding to the pair $(V',V)$.
	For the induction step assume the statement holds for $k \ge 2$ and let us prove it for $k+1$. Let $\P$ be a $k$-partition on $(\Vside^1,\ldots,\Vside^{k+1})$, let $F \in E_{k+1}(\P)$ and let $V \in V_{k+1}(\P)$,
	and denote the vertex classes of $\F$ by $(V_1,\ldots,V_k)$ with $V_j \sub \Vside^j$ for every $1 \le j \le k$.
	Recall that, by Definition~\ref{def:r-partition},
	$\P^{(\k)} \prec \K_{\k}(\P)$.
	Thus, $F \sub \K(G_1,\ldots,G_\k)$ with $G_j \in \P^{(\k-1)}$ for every $1 \le j \le \k$, where $G_j$ is a $(\k-1)$-partite $(\k-1)$-graph on $(V_1,\ldots,V_{j-1},V_{j+1},\ldots,V_\k)$.
	We have
	\begin{equation}\label{eq:decompose}
	F \circ V = \K(G_1 \circ V, \ldots,\, G_{\k} \circ V,\, F) \;;
	\end{equation}
	indeed, the  inclusion~$\sub$ follows from the fact that $F \sub \K(G_1,\ldots,G_\k)$, and the reverse inclusion~$\supseteq$ is immediate.
	Now, for every $1 \le j \le \k$, let $\P_j$ denote the restriction of $\P$ to the vertex classes $(\Vside^1,\ldots,\Vside^{j-1},\Vside^{j+1},\ldots,\Vside^\k,\Vside^{k+1})$
	and apply the induction hypothesis with the $(\k-1)$-partition $\P_j$, the $(k-1)$-graph $G_j$ and $V$. 
	It follows that there is a partition $G_j \circ V = \bigcup_i \K(P_{j,i})$ where each $P_{j,i}$ is a $\k$-polyad of $\P_j$ (and thus of $\P$) on $(V_1,\ldots,V_{j-1},V_{j+1},\ldots,V_\k,V)$.
	Since $\K_{\k}(\P_j) \succ \P_j^{(\k)}$, again by Definition~\ref{def:r-partition},
	for each $i$ and $j$ we have a partition $\K(P_{j,i}) = \bigcup_\ell F_{j,i,\ell}$ with $F_{j,i,\ell} \in \P_j^{(\k)}$, where each $F_{j,i,\ell}$ is a $\k$-partite $\k$-graph on $(V_1,\ldots,V_{j-1},V_{j+1},\ldots,V_\k,V)$.
	Summarizing, for every $1 \le j \le k$ we have the partition $G_j \circ V = \bigcup_{i,\ell} F_{j,i,\ell}$,
	and so it follows using~(\ref{eq:decompose}) that we have the partition
	$$
	F \circ V = \K\Big(\bigcup_{i,\ell} F_{1,i,\ell}, \ldots,\, \bigcup_{i,\ell} F_{\k,i,\ell},\, F \Big)
	= \bigcup_{\substack{i_1,\ldots,i_\k\\\ell_1,\ldots,\ell_\k}} \K(F_{1,i,\ell},\ldots,\,F_{\k,i,\ell},\, F) \;.
	$$
	As each $(\k+1)$-tuple $(F_{1,i,\ell},\ldots,\,F_{\k,i,\ell}, F)$ corresponds to a $(\k+1)$-polyad of $\P$, this completes the inductive step.
\end{proof}

Before proving Claim~\ref{claim:uniform-refinement} we will also need the following two easy claims.
\begin{claim}\label{claim:star-union}
	Let $G_1,\ldots,G_\ell$ be mutually edge-disjoint bipartite graphs on the same vertex classes $(Z,Z')$.
	If every $G_i$ is $\langle \d \rangle$-regular then $G=\bigcup_{i=1}^\ell G_i$ is also $\langle \d \rangle$-regular.
\end{claim}
\begin{proof}
	Let $S \sub Z$, $S' \sub Z'$ with $|S| \ge \d|Z|$, $|S'| \ge \d|Z'|$.
	Then
	$$d_G(S,S') = \frac{e_G(S,S')}{|S||S'|}
	= \sum_{i=1}^\ell \frac{e_{G_i}(S,S')}{|S||S'|}
	= \sum_{i=1}^\ell d_{G_i}(S,S')
	\ge \sum_{i=1}^\ell \frac12 d_{G_i}(Z,Z')
	= \frac12 d_{G}(Z,Z') \;,$$
	where the second and last equalities follows from the mutual disjointness of the $G_i$, and the inequality follows from the $\langle \d \rangle$-regularity of each $G_i$.
	Thus, $G$ is $\langle \d \rangle$-regular, as claimed.
\end{proof}

\begin{claim}\label{claim:refinement-union}
	If $\Q \prec_\d \P$ then there exist $P \in \P$ and $Q$ that is a union of members of $\Q$ such that $|P \triangle Q| \le 3\d|P|$.
\end{claim}
\begin{proof}
	For each $P\in \P$ let $\Q(P) = \{Q \in \Q \colon Q \sub_\d P\}$,
	and denote $P_\Q = \bigcup_{Q \in \Q(P)} Q$.
	We have
	\begin{align*}
	\sum_{P \in \P} |P \triangle P_\Q|
	&= \sum_{P \in \P} |P_\Q \sm P| + \sum_{P \in \P} |P \sm P_\Q|
	= \sum_{P \in \P} \sum_{\substack{Q \in \Q \colon\\Q \sub_\d P}} |Q \sm P|
	+ \sum_{P \in \P} \sum_{\substack{Q \in \Q \colon\\Q \nsubseteq_\d P}} |Q \cap P| \\
	&\le \sum_{P \in \P} \sum_{\substack{Q \in \Q \colon\\Q \sub_\d P}} \d|Q|
	+ \Big( \sum_{\substack{Q \in \Q\colon\\Q \notin_\d \P}} |Q|
	+ \sum_{\substack{Q \in \Q \colon\\Q \in_\d \P}} \d|Q| \Big)
	\le 3\d\sum_{Q \in \Q} |Q|
	= 3\d\sum_{P \in \P} |P| \;,
	\end{align*}
	where the last inequality uses the statement's assumption $\Q \prec_\d \P$ to bound the middle summand.
	By averaging, there exists $P \in \P$ such that $|P \triangle P_\Q| \le 3\d|P|$, thus completing the proof.
\end{proof}

\paragraph*{Proofs of properties.}
We are now ready to prove the properties of $k$-partitions stated at the beginning of Section~\ref{sec:LB}.

\renewcommand{\k}{r}

\begin{proof}[Proof of Claim~\ref{claim:uniform-refinement}]
	Put $\E = E_k(\P)$, and let us henceforth use $\k=k-1$. Since $\E \prec_\d \F$, Claim~\ref{claim:refinement-union} implies that there exist $F \in \F$ (an $\k$-partite $\k$-graph on $(\Vside^1,\ldots,\Vside^{\k})$), as well as an $\k$-partite $\k$-graph $F_\E$ that is a union of members of $\E$, such that $|F \triangle F_\E| \le 3\d|F|$.
	Denote by $\Q$ the $(\k-1)$-partition $\P'$ obtained by restricting $\P$ to $\bigcup_{i=1}^{\k} V_i$, and note that $\Q$ is $\langle \d \rangle$-good since $\Q \sub \P$ and, by assumption, $\P$ is $\langle \d \rangle$-good.
	Our goal is to prove that $\Q$ is a $\langle 3\d \rangle$-regular partition of $F$.
	Recalling Definition~\ref{def:k-reg}, note that it suffices to show, without loss of generality, that $E_{\k}(\Q) \cup V_{\k}(\Q)$ is a $\langle \d \rangle$-regular partition of the bipartite graph $G_{F}^{\k}$. 
	We have $|G_{F}^{\k} \triangle G_{F_\E}^\k|=|F \triangle F_\E| \le 3\d|F|=3\d|G_{F}^{\k}|$, that is, $G_{F_\E}^\k$ is obtained from $G_{F}^{\k}$ by adding/removing at most $3\d |G_{F}^{\k}|$ edges.
	Therefore, to complete the proof it suffices to show that for every $Z \in E_{\k}(\Q)$ and $Z' \in V_{\k}(\Q)$, the induced bipartite graph $G_{F_\E}^\k[Z,Z']$ is $\langle \d \rangle$-regular (recall Definition~\ref{def:star-regular}).
	
	Apply Claim~\ref{claim:decomposition} on the $(\k-1)$-partition $\Q$ with $Z$ and $Z'$. Since $\E \prec \K_{\r-1}(\Q)$ (recall Definition~\ref{def:r-partition}),
	this means that $Z \circ Z'$ is a (disjoint) union of members $E$ of $\E$ all underlied by $\k$-polyads of the form  $(P_1,\ldots,P_{\k-1},Z)$.
	Since $\Q$ is $\langle \d \rangle$-good (recall Definition~\ref{def:k-good}), for each such $E$ we in particular have that $G^\k_E[Z,Z']$ ($=G^\k_{E,\,\U(E)}$) is $\langle \d \rangle$-regular.
	It follows that $G_{F_\E}^\k[Z,Z']$ is a disjoint union of $\langle \d \rangle$-regular bipartite graphs on $(Z,Z')$.
	Claim~\ref{claim:star-union} thus implies that $G_{F_\E}^\k[Z,Z']$ is a $\langle \d \rangle$-regular bipartite graph. As explained above, this completes the proof.
\end{proof}

We end this subsection with the easy proof of Claim~\ref{claim:restriction}.
\begin{proof}[Proof of Claim~\ref{claim:restriction}]
	Recall Definition~\ref{def:k-reg}.
	Clearly, $\P'$ is $\langle \d \rangle$-good. We will show that $E_1(\P') \cup V_1(\P')$ is a $\langle \d/\b \rangle$-regular partition of $G^1_{H'}$.
	The argument for $G^i_{H'}$ for every $2 \le i \le k$ will be analogous, hence the proof would follow.
	Observe that $G^1_{H'}$ is an induced subgraph of $G^1_{H}$, namely, $G^1_{H'} = G^i_{H}[V_2' \times\cdots\times V_k',\, V_1']$.
	By assumption, $e(H') = \b e(H)$, and thus $e(G^1_{H'}) = \b e(G^1_{H})$.
	By the statement's assumption on $\P^{(1)}$ and since $E_1(\P) \cup V_1(\P)$ is a $\langle \d \rangle$-regular partition of $G^1_{H}$, we deduce---by adding/removing at most $\d e(G^1_{H}) =  (\d/\b)e(G^1_{H'})$ edges of $G^1_{H'}$---that $E_1(\P') \cup V_1(\P')$ is a $\langle \d/\b \rangle$-regular partition of $G_{H'}^1$.
	As explained above, this completes the proof.
\end{proof}




\section{Ackermann-type Lower Bounds for the R\"odl-Schacht Regularity Lemma}\label{sec:coro}


\renewcommand{\K}{\mathcal{K}}

The purpose of this section is to apply Theorem~\ref{theo:main} in order to prove Corollary~\ref{coro:RS-LB}, giving level-$k$ Ackermann-type lower bounds for the $k$-graph regularity lemma of R\"odl-Schacht~\cite{RodlSc07}.
We begin with the required definitions.
The definitions we state here are essentially equivalent to (though shorter than) those in~\cite{RodlSc07}.
We will rely on the definitions in Subsection~\ref{subsec:preliminaries}, and in particular, the definition of a $k$-partition.

For a $k$-graph $H$, the \emph{density} of a $(k-1)$-graph $S$ in $H$ is
$$d_H(S) = \frac{|H \cap \K(S)|}{|\K(S)|} \;,$$
%
%
where $d_H(S)=0$ if $|\K(S)|=0$.
The notion of $\e$-regularity for $k$-graphs is defined as follows.

\begin{definition}[$\e$-regular $k$-graph]\label{def:e-reg}
	A $k$-partite $k$-graph $H$ is \emph{$(\e,d)$-regular}---or simply \emph{$\e$-regular}---in a $k$-polyad $P$ with $H \sub \K(P)$
	if for every $S \sub P$ with $|\K(S)| \ge \e|\K(P)|$ we have $d_H(S) = d \pm \e$.
\end{definition}

A \emph{partition} of a $k$-graph $H$ is simply a $(k-1)$-partition on $V(H)$.

\begin{definition}[$\e$-regular partition]\label{def:e-reg-partition}
	A partition $\P$ of a $k$-graph $H$ is \emph{$\e$-regular} if 
	$\sum_P |\K(P)| \le \e|V(H)|^k$ where the sum is over all $k$-polyads $P$ of $\P$ for which
	$H \cap \K(P)$ is not $\e$-regular in $P$.	
\end{definition}

Henceforth, an \emph{$(r,a_1,\ldots,a_r)$-partition} is simply an $r$-partition $\P$ (recall Definition~\ref{def:r-partition}) where  $|\P^{(1)}|=a_1$ and for every $2 \le i \le r$,
$\P^{(i)}$ subdivides each $K \in \K_i(\P)$ into $a_i$ parts.

\begin{definition}[$f$-equitable partition]\label{def:r-equitable}
	Let $f\colon[0,1]\to[0,1]$.
	An $(r,a_1,\ldots,a_r)$-partition $\P$ is \emph{$f$-equitable} if $\P^{(1)}$ is equitable and for every $2 \le i \le r$, every $i$-graph $F \in \P^{(i)}$ is $(\e,1/a_i)$-regular in $\U(F)$, where $\e=f(d_0)$ and $d_0=\min\{1/a_2,\ldots,1/a_{r}\}$.
\end{definition}

\subsection{The lower bound}

The $k$-graph regularity of R\"odl-Schacht~\cite{RodlSc07} states, roughly, that for every $\e>0$ and every function $f\colon\N\to(0,1]$, every $k$-graph has an $\e$-regular $f$-equitable equipartition $\P$ where $|\P|$ is bounded by a level-$k$ Ackermann-type function.
In fact, R\"odl-Schacht's $k$-graph regularity lemma (Theorem~17 in~\cite{RodlSc07}) uses a considerably stronger notion of regularity of a partition that involves an additional function $r$ which we shall not discuss here (this stronger notion was crucial in~\cite{RodlSc07} for allowing them to prove a counting lemma).
Our lower bound below applies even to the weaker notion stated above, which corresponds to taking $r \equiv 1$.

The proof of  Corollary~\ref{coro:RS-LB} 
will follow quite easily from Theorem~\ref{theo:main} together with Claim~\ref{claim:k-reduction} below.
Claim~\ref{claim:k-reduction} basically shows that a $\langle \d \rangle$-regularity ``analogue'' of R\"odl-Schacht's notion of regularity implies graph $\langle \d \rangle$-regularity. 
%
It is essentially a generalization of a similar claim from~\cite{MS3}.
The proof of Claim~\ref{claim:k-reduction} is deferred to the Appendix~\ref{sec:RS-appendix}.
Henceforth we say that a graph partition is \emph{perfectly $\langle \d \rangle$-regular} if all pairs of distinct clusters are $\langle \d \rangle$-regular without modifying any of the graph's edges.
\begin{claim}\label{claim:k-reduction}
	Let $H$ be a $k$-partite $k$-graph on vertex classes $(\Vside^1,\ldots,\Vside^k)$, and
	let $\P$ be an $f$-equitable partition of $H$ with $\P^{(1)} \prec \{\Vside^1,\ldots,\Vside^k\}$,
	$f(x) = \d^4(x/2)^{2^{k+3}}$ and $|V(H)| \ge n_0(\d,|\P|)$.
	Suppose that for each $k$-polyad $P$ of $\P$, every $S \sub P$ with $|\K(S)| \ge \d|\K(P)|$ has $d_{H}(S) \ge \frac23 d_{H}(P)$.
	Then $E_k(\P) \cup V_k(\P)$ is a perfectly $\langle 2\sqrt{\d} \rangle$-regular 
	partition of $G_H^k$.
\end{claim}

We now formally restate and prove Corollary~\ref{coro:RS-LB}.
%
We mention that, as will be immediate from the proof, our lower bound not only applies to the hypergraph regularity lemma of R\"odl and Schacht but also to the hypergraph regular approximation lemma~\cite{RodlSc07}.

\begin{theo}[Lower bound for R\"odl-Schacht's $k$-graph regularity lemma]\label{theo:RS-LB}
	Let $s \ge k \ge 2$ and
	put $c = 2^{-32^k}$.
	For every $s \in \N$ there exists a $k$-partite $k$-graph $H$ of density $p=2^{-s-k}$, and a partition $\V_0$ of $V(H)$ with $|\V_0| \le k 2^{200}$,
	such that if $\P$ is an $\e$-regular $f$-equitable partition of $H$ 	
	with
	$\e \le c p$,
	$f(x) \le c^4(x/2)^{2^{k+3}}$, $|V(H)| \ge n_0(k,|\P|)$
	and $\P^{(1)} \prec \V_0$, then $|\P^{(1)}| \ge \Ack_k(s)$.
\end{theo}

\begin{remark}
	One can easily remove the assumption $\P^{(1)} \prec \V_0$ by replacing $\P$ with its appropriate intersection
	with $\V_0$. Since $|\V_0|=O(k)$ this has only a minor effect on the parameters of $\P$ and thus
	one gets essentially the same lower bound. We omit the details of this routine transformation.
\end{remark}


\begin{proof}[Proof of Theorem~\ref{theo:RS-LB}]
	Put $\a_k = 2^{-16^k}$ (recall $c = 2^{-32^k}$).
	The bound $|\P^{(1)}| \ge \Ack_k(s)$ would follow from Theorem~\ref{theo:main} if we show that $H$ is $\langle \a_k \rangle$-regular relative to the $(k-1)$-partition $\P$.
	Henceforth put $\d=(\a_k/2)^2$. We will later use the inequalities
	\begin{equation}\label{eq:RS-LB-ineq}
	c \le \frac{\a_k^k}{7k^k} \le \d \;.
	\end{equation}
	First we claim that $\P$ is $\langle \a_k \rangle$-good (recall Definition~\ref{def:k-good}).
	Let $2 \le r \le k-1$, let $F \in \P^{(r)}$ be an $r$-partite $r$-graph on $(V_1,\ldots,V_r)$, and denote by $P=\U(F)$ the $r$-polyad underlying $F$.	
	We will show that the bipartite graph $G_{F,P}^r$ is $\langle \a_k \rangle$-regular, and since an analogous argument will hold for $G_{F,P}^i$ for every $1 \le i \le r$, this would prove our claim.
	Recalling Definition~\ref{def:aux}, we have that $G_{F,P}^r = G_F^r[E,V_r]$ with $E:=P[V_1,\ldots,V_{r-1}] \in E_r(\P)$.
	Now, suppose $\P$ is a $(k-1,a_1,\ldots,a_{k-1})$-partition, put $d_r = 1/a_r$ for each $2 \le r \le k-1$, and put $d_0=\min\{1/a_2,\ldots,1/a_{k-1}\}$. 
	Recalling Definition~\ref{def:r-equitable}, since $\P$ is $f$-equitable we have that $F$ is $(f(d_0),d_r)$-regular in $P$. Thus, recalling Definition~\ref{def:e-reg}, for every $S \sub P$ with $|K(S)| \ge \d|K(P)|$ (note $\d \ge c \ge f(d_0)$ using~(\ref{eq:RS-LB-ineq})) we have $d_F(S) \ge d_r-f(d_0) \ge d_F(P) - 2f(d_0) \ge \frac23 d_F(P)$.
	Let the $(r-1)$-partition $\P'$ be obtained by restricting $\P$ to $V_1 \cup\cdots \cup V_r$ (so in particular $V_i(\P')=\{V_i\}$), and note that $P$ is an $r$-polyad of $\P'$.
	Observe that $d_F(S) \ge \frac23 d_F(P)$ trivially holds for any $r$-polyad $P'$ of $\P'$ other than $P$ as well, since $d_F(P')=0$ as $F \sub \K(P)$.
	Apply Claim~\ref{claim:k-reduction}, with (the almost trivial choice of) the $r$-partite $r$-graph $F$ and the $f$-equitable $(r-1)$-partition $\P'$ of $F$, to deduce that $E_{r}(\P') \cup \{V_r\}$ 
	is a perfectly $\langle\a_k\rangle$-regular (i.e., $\langle 2\sqrt{\d} \rangle$-regular) partition of $G_{F}^r$.
	Since $E \in E_{r}(\P')$, this in particular implies that $G_F^r[E,V_r]$ is $\langle\a_k\rangle$-regular, which proves our claim as explained above.
	
	
	%
	
	It remains to show that $H$ is $\langle \a_k \rangle$-regular relative to the  $\langle \a_k \rangle$-good $(k-1)$-partition $\P$.
	Let $H'$ be obtained from $H$ by removing all its ($k$-)edges underlied by $k$-polyads of $\P$ such that either $d_H(P) \le 6\e$ or the $k$-graph $H \cap \K(P)$ is not $\e$-regular in $P$.
	By Definition~\ref{def:e-reg-partition}, the number of edges removed from $H$ to obtain $H'$ is at most
	$$\e|V(H)|^k + 6\e|V(H)|^k \le 7\cdot c p |V(H)|^k
	\le (\a_k p/k^k)|V(H)|^k = \a_k\cdot e(H) \;,$$
	where the inequalities use the statement's assumption on $\e,c$ and~(\ref{eq:RS-LB-ineq}), and the equality uses the fact that all $k$ vertex classes of $H$ are of the same size (see Remark~\ref{remark:main}).
	Thus, in $H'$, every non-empty $k$-polyad of $\P$ is $\e$-regular and of density at least $6\e$.
	Again by Definition~\ref{def:e-reg-partition}, for every $k$-polyad $P$ of $\P$ and every $S \sub P$ with $|\K(S)| \ge \d |\K(P)|$ ($\ge \e |\K(P)|$
	by~(\ref{eq:RS-LB-ineq})) we have $d_H(S) \ge d_H(P)-2\e \ge \frac23 d_H(P)$.
	Apply Claim~\ref{claim:k-reduction}, this time with $H$ and $\P$.
	It follows that 
	$E_{k}(\P) \cup V_{k}(\P)$
	is an $\langle \a_k \rangle$-regular 
	partition of $G_H^k$. An analogous argument applies for $G_H^i$ for every $1 \le i \le k$, thus completing the proof.
\end{proof}

\section{$\langle \d \rangle$-regularity does not suffice for triangle counting}\label{sec:example}

Here we construct for every fixed $\delta>0$ and small enough $p$, a graph of density $p$ which is $\langle \d \rangle$-regular yet does not
contain the expected number of triangles 
(actually, the example is going to be triangle free).
This shows that  $\langle \d \rangle$-regularity, even just for graphs, is strictly weaker than Szemer\'edi's regularity.
The precise statement is the following.

\begin{prop}\label{claim:example}
For every $0 < p \le 10^{-3}\d^5$ and large enough $n$ there is a $n$-vertex tripartite graph of density at least $p$, whose every pair of classes span a $\langle \d \rangle$-regular graph, and yet is triangle free.
\end{prop}
We use the following well-known lemma, where we denote by $\norm{v}_1$ the  $\ell_1$-norm of a vector $v$ (for a proof see, e.g., Lemma 4.3 in~\cite{KaliSh13}).
\begin{lemma}\label{lemma:convex}
	Every vector $x \in [0,1]^n$ is a convex combination of binary vectors $y \in \{0,1\}^n$ each with $\norm{y}_1=\norm{x}_1$.
%
%
%
	%
\end{lemma}

We will also apply the following version of the Chernoff bound.
\begin{lemma}[Multiplicative Chernoff bound]\label{lemma:Chernoff}
	Let $X_1,\ldots,X_n$ be mutually independent Bernoulli random variables,
	and put $X=\sum_{i=1}^n X_i$, $\mu = \Ex[X]$.
	For every $\d \in [0,1]$ we have
	$$\Pr(X \neq (1 \pm \d)\mu) \le 2\exp\Big(-\frac13\d^2\mu\Big) \;.$$
\end{lemma}

\begin{proof}[Proof of Proposition~\ref{claim:example}]
	We will in fact construct a graph satisfying an even somewhat stronger property than $\langle \d \rangle$-regularity (namely, the constant $\frac12$ is Definition~\ref{def:star-regular} will be replaced by $1-\d$).
	Consider a random tripartite graph on vertex classes $(V_1,V_2,V_3)$, each of size $k$, obtained by independently retaining each edge with probability 
	$q:=3p$ $(\le 1)$, where $k$ is any integer satisfying
	\begin{equation}\label{eq:example-k}
	64\d^{-2}q^{-1} \le k
	\le \frac14\d^3 q^{-2} \;.
	\end{equation}
	Note that $k$ is well defined in~(\ref{eq:example-k}) by the statement's bound on $p$.
	%
	%
	Denoting by $X$ the random variable counting the triangles, one can easily check that
	$$\Exp[X] = k^3 q^3 \quad\text{ and }\quad
	\Var[X] \le k^3 q^3 + 3\binom{k}{2}k^2 q^5 \;.$$
	Chebyshev's inequality thus implies that
	$$\Pr[X \ge \d^3k^2 q] \le \frac{\Var[X]}{(\d^3k^2q - \Exp[X])^2}
	\le \frac{k^3q^3(1+\frac32 k q^2)}{k^4q^2(\d^3 - kq^2)^2}
	\le \frac{q}{k} \cdot 8\d^{-6}
	\le \frac18 q^2 \d^{-4}
	= \frac98 p^2 \d^{-4}
	< \frac12 \;,$$
	where the third inequality uses the upper bound $kq^2 \le \frac14 \d^3$ $(\le \frac14)$ from~(\ref{eq:example-k}), the fourth inequality uses the lower bound from~(\ref{eq:example-k}), and the last inequality uses the statement's bound on $p$.
	Next, by using Lemma~\ref{lemma:Chernoff}
	together with the union bound we deduce that 
	\begin{equation}\label{eq:example-property}
	\forall 1 \le a<b \le 3\,\,\,
	\forall S \sub V_a,\, T \sub V_b \text{ with } |S| \ge \d|V_a|,\, |T| \ge \d|V_b| \,\,\colon\,\,
	d(S,T) = \big(1\pm\frac13\d\big)q
	\end{equation}
	except with probability at most
	$$3 \cdot 2^{2k} \cdot 2\exp\Big(-\frac{1}{27}\d^2 qk^2\Big) \le \frac12 \;,$$
	where the inequality uses the lower bound $kq \ge 64\d^{-2}$ from~(\ref{eq:example-k}).
	We deduce from all of the above that there exists a tripartite graph that has $k$ vertices in each vertex class, at most $\d^3 k^2 q$ triangles and satisfies~(\ref{eq:example-property}).
	By removing an edge from each triangle, one-third of them from each of the three pairs of vertex classes, we obtain a triangle-free graph $G_0$ such that for every $a<b$ and subsets $S \sub V_a$, $T \sub V_b$
	with $|S|\ge\d|V_a|$, $|T| \ge \d|V_b|$ we have
	$$e(S,T) \ge \big(1-\frac13\d\big)q|S||T|-\frac13 \d^3k^2q \ge \big(1-\frac23\d\big)q|S||T|
	\ge \frac{1-\frac23\d}{1+\frac13\d}d(V_a,V_b)|S||T|
	\ge (1-\d)d(V_a,V_b)|S||T|$$
	where the first and third inequalities follows from the lower and upper bound in~(\ref{eq:example-property}), respectively.
	In particular, we deduce from the inequality $e(S,T) \ge (1-\frac23\d)q|S||T|$ above that $d(V_a,V_b) \ge \frac13 q=p$.
	Summarizing, $G_0$ is triangle free, has density at least $p$ and satisfies 
	\begin{equation}\label{eq:example-property2}
	\forall a<b\,\,
	\forall S \sub V_a,\, T \sub V_b \text{ with } |S| \ge \d|V_a|,\, |T| \ge \d|V_b| \,\,\colon\,\, d(S,T) \ge (1-\d)d(V_a,V_b) \;.
	\end{equation}

	To obtain from $G_0$ a graph on a large enough number of vertices we simply take a blow-up, replacing each vertex $v$ of $G_0$ by a set $G(v)$ of $m$ new vertices, for $m \in \N$ arbitrarily large. The resulting tripartite graph $G$ is clearly triangle free and of density $d(G_0) \ge p$.
	It thus remains to prove that any two vertex classes of $G$ span a bipartite graph satisfying the desired regularity property.
	Let $a<b$ and let $\overline{V_a},\overline{V_b}$ denote the vertex classes of $G$ corresponding to $V_a,V_b$.
	Note that $|V_a|=|V_b|=k$ and $|\overline{V_a}|=|\overline{V_b}|=mk$.
	Let $S \sub \overline{V_a}$, $T \sub \overline{V_b}$ with $|S| \ge \d|\overline{V_a}|$ and $|T| \ge \d|\overline{V_b}|$.
	We will show that 
	\begin{equation}\label{eq:example-goal}
	d(S,T) \ge (1-\d)d(\overline{V_a},\overline{V_b}) \;,
	\end{equation}
	which would complete the proof.
	Consider the two vectors $s,t \in [0,1]^k$ defined as follows:
	$$s=(|S \cap G(u)|/m)_{u \in \overline{V_a}} \quad\text{ and }\quad
	t=(|T \cap G(v)|/m)_{v \in \overline{V_b}} \;.$$
	Note that $\norm{s}_1 = |S|/m \ge \d k$ and $\norm{t}_1 = |T|/m \ge \d k$.
	By Lemma~\ref{lemma:convex} applied on $s$,
	$$s = \sum_i \a_i s_i \quad\text{ with }\quad s_i \in \{0,1\}^k,\, \norm{s_i}_1 = |S|/m \text{ and } \a_i \ge 0,\, \sum_i \a_i = 1 \;.$$
	By Lemma~\ref{lemma:convex} applied on $t$,
	$$t = \sum_j \b_j t_j \quad\text{ with }\quad t_j \in \{0,1\}^k,\, \norm{t_j}_1 = |T|/m \text{ and } \b_j \ge 0,\, \sum_j \b_j = 1 \;.$$		
	%
	%
	Let $A$ denote the $k \times k$ bi-adjacency matrix of $G_0[V_a,V_b]$.
	Observe that $e_G(S,T) = (ms)^T A (mt)$.
	Moreover, observe that for every $i,j$ we have that $s_i^T A t_j$ is the number of edges of $G_0$ between the subsets of $V_a,V_b$ corresponding to $s_i,t_j$, respectively. Note that these subsets are of size $\norm{s_i}_1,\norm{t_j}_1$, respectively, which are both at least $\d k$.
	Thus, by~(\ref{eq:example-property2}),	
	$$s_i^T A t_j \ge (1-\d)d(\overline{V_a},\overline{V_b})\norm{s_i}_1\norm{t_j}_1
	= (1-\d)d(\overline{V_a},\overline{V_b})|S||T|/m^2 \;.$$
%
	We deduce that
	\begin{align*}
	e(S,T) &= m^2 \cdot s^T A t
	= m^2\Big(\sum_i \a_i s_i \Big)^T A \Big(\sum_j \b_j t_j\Big)
	= m^2\sum_{i,j} \a_i\b_j s_i^T A t_j\\
	&\ge \Big(\sum_i \a_i \Big)\Big(\sum_j \b_j \Big)(1-\d) d(\overline{V_a},\overline{V_b})|S||T| = (1-\d)d(\overline{V_a},\overline{V_b})|S||T|\;.
	\end{align*}
%
	This gives~(\ref{eq:example-goal}) and thus completes the proof.
\end{proof}

\appendix

\section{Proof of Claim~\ref{claim:k-reduction}}\label{sec:RS-appendix}



\renewcommand{\K}{\mathcal{K}}

\subsection{Basic facts}

In order to prove Theorem~\ref{theo:RS-LB} we will need several auxiliary results and definitions.
We begin with the notion of complexes.
Henceforth, the \emph{rank} of a (not necessarily uniform)
hypergraph $P$ is $\max_{e \in P} |e|$.
For $r \ge 2$ we denote 
$P^{(r)} = \big\{e \in P \,\big\vert\, |e|=r\big\}$ and $P^{(1)}=V(P)$.

\begin{definition}[complex]
	A \emph{$k$-complex} $(k \ge 2)$ is a $k$-partite hypergraph $P$ of rank $k-1$ where $P^{(r)} \sub \K(P^{(r-1)})$ for every $2 \le r \le k-1$.
\end{definition}


\begin{definition}[$f$-regular complex]\label{def:complex}		
	Let $f \colon [0,1]\to[0,1]$. A $k$-complex $P$ on vertex classes $(V_1,\ldots,V_k)$ is \emph{$(f,d_2,\ldots,d_{k-1})$-regular}, or simply \emph{$f$-regular}, if
	for every $2 \le r \le k-1$ and every $r$ vertex classes $V_{i_1},\ldots,V_{i_r}$ we have that $P^{(r)}[V_{i_1},\ldots,V_{i_r}]$ is $(\e,d_r)$-regular in $P^{(r-1)}[V_{i_1},\ldots,V_{i_r}]$,
	where $\e=f(d_0)$ and $d_0=\min\{d_2,\ldots,d_{k-1}\}$.
	%
\end{definition}

Note that by using the notion of complexes one can equivalently define an $f$-equitable partition (recall Definition~\ref{def:r-equitable}) as follows;
	an $(r,a_1,\ldots,a_r)$-partition $\P$ is $f$-equitable if $\P^{(1)}$ is equitable and, if $r \ge 2$, every $r$-complex
	of $\P$
	is $(f,1/a_2,\ldots,1/a_r)$-regular.

We now state the \emph{dense counting lemma} of~\cite{RodlSc07-B} specialized to
complexes.
%
We henceforth fix the following notation for $k \ge 3$, $\g > 0$;
\begin{equation}\label{eq:DCL}
F_{k,\g}(x) := \frac{\g^3}{12}\Big(\frac{x}{2}\Big)^{2^{k+1}} \;.
\end{equation}
The statement we use below follows from combining Theorem~10 and Corollary~14 in~\cite{RodlSc07-B}, and generalized to the case where the vertex classes are not necessarily of the same size.
For a $k$-polyad $F$ and an edge $e \in F$, we denote the set of $k$-cliques in $F$ containing $e$ by $\K(F,e)=\{e' \in \K(F) \,\vert\, e \sub e'\}$.
For a $k$-complex $P$ we abbreviate $\K(P):=\K(P^{(k-1)})$.
\begin{fact}[Dense counting lemma for $k$-complexes]\label{fact:counting}
	Let $\g > 0$ and let $P$ be a $k$-complex $(k \ge 3)$
	that is $(F_{k,\g},d_2,\ldots,d_{k-1})$-regular with $n_i \ge n_0(\g,d_2,\ldots,d_{k-1})$ vertices in the $\ith$ vertex class.
	Then	
	$$|\K(P)| = (1 \pm \g)\prod_{i=2}^{k-1} d_i^{\binom{k}{i}} \cdot \prod_{i=1}^k n_i \;.$$
	Moreover, denoting $P^{(k-1)}=(P_1,\ldots,P_k)$, we have for all edges $e \in P_k$ but at most $\g|P_k|$ that
	$$|\K(P,e)| = (1 \pm \g)\prod_{i=2}^{k-1} d_i^{\binom{k-1}{i-1}} \cdot n_{k} \;.\footnote{In~\cite{RodlSc07-B}, the statement of the `moreover' part (Corollary 14, dense extension lemma) allows for $\g|P^{(k-1)}|$ exceptional edges in $P^{(k-1)}$ rather than only in $P_k$, which is nevertheless essentially equivalent to our statement. Furthermore, they allow for counting not only $k$-cliques, in which case they do not need all $P_i$ to be regular.}
	\footnote{To obtain the bound $F_{k,\g}$ from the proof of Corollary~14 in~\cite{RodlSc07-B} (with $h=k-1$ and $\ell=k$), one can verify that:
		\begin{itemize}
		\item $\e(\mathbf{\F},\g,d_0)$ in Theorem~13 can be bounded by $\g(d_0/2)^{|\F^{(h)}|}$, and so $\e(K_{k}^{(k-1)},\g,d_0) \le \g(d_0/2)^{2^k}$,
		\item $\b$ in Fact~15 can be bounded by $\g^3/4$,
		\item $\e_{GDCL}(\D(\F^{(h)},f),\frac{\b}{3},d_0)$ in the proof of Corollary~14 can be bounded by $\frac{\b}{3}(d_0/2)^{2^{k+1}}$, using the first item and the fact that $\D(\F^{(h)},f)$ has at most $2k-(k-1)=k+1$ vertices.
		\end{itemize}}$$
	%
\end{fact}



We will also need a \emph{slicing lemma} for complexes.

\begin{lemma}[Slicing lemma for complexes]\label{lemma:k-slice}
	Let $P$ be a $k$-complex $(k\ge 3)$ on vertex classes $(V_1,\ldots,V_k)$ and let $V_k' \sub V_k$ with $|V_k'| \ge \d|V_k|$.
	If $P$ is $(f,\,d_2,\ldots,d_{k-1})$-regular with
	$f(x) \le \frac{\d}{2} F_{k-1,\frac14}(x)$
	and $|V(P)| \ge n_0(d_2,\ldots,d_{k-1})$
	then the induced $k$-complex $Q=P[V_1,\ldots,V_{k-1},V_k']$ is $(f^*,d_2,\ldots,d_{k-1})$-regular with $f^*=\frac{2}{\d} \cdot f$.
\end{lemma}
For the proof we will need the notation $P^{(\le i)} = \{e \in P \,\vert\, |e| \le i\}$ where $P$ is any hypergraph.
\begin{proof}
	We proceed by induction on $k$. We begin with the induction basis $k=3$.
	Let $P=(P_1,P_2,P_3)$ be an $(f,d_2)$-regular $3$-complex on vertex classes $(V_1,V_2,V_3)$, meaning that each bipartite graph $P_i$ (which is obtained from $P$ by removing $V_i$ and its adjacent edges)
	is $(\e,d)$-regular with $d=d_2$ and $\e=f(d)$.
	Denoting $Q=(Q_1,Q_2,Q_3)$,
	we will show that the bipartite graphs $Q_1=P_1[V_2,V_3']$ and $Q_2=P_2[V_1,V_3']$ are each $(\e/\d,\,d)$-regular.
	Since $f(x)/\d \le f^*(x)$, and since $Q_3=P_3$ is $(\e,d)$-regular by assumption, this would imply that $Q$ is $(f^*,d_2)$-regular, as needed.
	To prove that $Q_1$ is $(\e/\d,d)$-regular, let $S \sub V_2 \cup V_3'$ with $|\K(S)| \ge (\e/\d)|V_2||V_3'|$. Then $|\K(S)| \ge \e|V_2||V_3|$, hence $d_{Q_1}(S)=d_{P_1}(S) = d \pm \e$, as desired.
	Similarly, to prove that $Q_{2}$ is $(\e/\d,d)$-regular, let $S \sub V_1 \cup V_3'$ with $|\K(S)| \ge (\e/\d)|V_1||V_3'|$. Then $|\K(S)| \ge \e|V_1||V_3|$, hence $d_{Q_2}(S)=d_{P_2}(S) = d \pm \e$. This proves the induction basis.

	It remains to prove the induction step.
	Let $P$ be a $(k+1)$-complex on vertex classes $(V_1,\ldots,V_{k+1})$ and let $V_{k+1}' \sub V_{k+1}$ with $|V_{k+1}'| \ge \d|V_{k+1}|$, and suppose $P$ is $(f,d_2,\ldots,d_k)$-regular with
	\begin{equation}\label{eq:k-slice-f-bound}
	f(x) \le \frac{\d}{2} F_{k,\frac14}(x) \;.
	\end{equation}
	We need to show that the induced $(k+1)$-complex $Q=P[V_1,\ldots,V_k,V_{k+1}']$ is $(f^*,d_2,\ldots,d_k)$-regular.
	Put $d_0 = \min\{d_2,\ldots,d_{k-1}\}$,
	$P^{(k)}=(P_1,\ldots,P_{k+1})$ and $Q^{(k)}=(Q_1,\ldots,Q_{k+1})$.
	Let $i \in [k+1]$, and observe that the regularity assumption on $P$ translates to the following assumptions on $P_i$:
	\begin{enumerate}
		\item the $k$-complex $P_i^{(\le k-1)}$ is $(f,d_2,\ldots,d_{k-1})$-regular,
		\item the $k$-partite $k$-graph $P^{(k)}_i$ is $(f(d_0),\,d_k)$-regular in $P^{(k-1)}_i$.
	\end{enumerate}	
	To prove that $Q$ is $(f^*,d_2,\ldots,d_k)$-regular
	we need to show that $Q_i$ satisfies the following conditions:
	\begin{enumerate}
		\item the $k$-complex $Q_i^{(\le k-1)}$ is $(f^*,d_2,\ldots,d_{k-1})$-regular,
		\item the $k$-partite $k$-graph $Q^{(k)}_i$ is $(f^*(d_0),\,d_k)$-regular in $Q^{(k-1)}_i$.
	\end{enumerate}
	We henceforth assume $i \neq k+1$, since otherwise $Q_i=P_i$ and so the above conditions follow from the above assumptions together with the fact that $f(x) \le f^*(x)$.
	Apply the induction hypothesis with the $k$-complex $P_i^{(\le k-1)}$ and $V_{k+1}'$, using assumption~$(i)$, the fact that $f(x) \le \frac{\d}{2}F_{k-1,\frac14}(x)$ by~(\ref{eq:k-slice-f-bound}) and the statement's assumption on $|V(P)|$.
	It follows that the $k$-complex $Q_i^{(\le k-1)}=P_i^{(\le k-1)}[V_1,\ldots,V_{i-1},V_{i+1},\ldots,V_k,V_{k+1}']$ is $(f^*,d_2,\ldots,d_k)$-regular,
	thus proving condition~$(i)$.
	
	Apply Fact~\ref{fact:counting} (dense counting lemma) with $\g=1/2$ and the $k$-complex $P_i^{(\le k-1)}$, using assumption~$(i)$, the fact that $f(x) \le F_{k,\frac12}(x)$ by~(\ref{eq:k-slice-f-bound}) and the statement's assumption on $|V(P)|$, to deduce that
	$$|\K(P_i^{(k-1)})| \le \frac32\prod_{j=2}^{k-1} d_j^{\binom{k}{j}} \cdot \prod_{\substack{1 \le j \le k+1\colon\\j \neq i}} |V_j| \;.$$
	On the other hand, applying Fact~\ref{fact:counting} with $\g=1/4$ and the $k$-complex $Q_i^{(\le k-1)}$, using condition~$(i)$, the fact that $f^*(x) = \frac{2}{\d}f(x) \le F_{k,\frac14}(x)$ by~(\ref{eq:k-slice-f-bound})
	and the statement's assumption on $|V(P)|$, implies that
	\begin{equation}\label{eq:k-slice-count}
	|\K(Q_i^{(k-1)})| \ge \frac34\prod_{j=2}^{k-1} d_j^{\binom{k}{j}} \cdot \prod_{\substack{1 \le j \le k\colon\\j \neq i}} |V_j| \cdot \d |V_{k+1}| \ge \frac{\d}{2}|\K(P_i^{(k-1)})| \;.
	\end{equation}
	We now prove condition~$(ii)$.
	Let $S \sub Q_i^{(k-1)}$ satisfy $|\K(S)| \ge f^*(d_0)|\K(Q_i^{(k-1)})|$.
	Then $|\K(S)| \ge f(d_0)|\K(P_i^{(k-1)})|$ by~(\ref{eq:k-slice-count}).
	Therefore $d_{Q_i^{(k)}}(S)=d_{P_i^{(k)}}(S) = d_k \pm f(d_0)$, where the last equality uses assumption~$(ii)$.
	This proves condition~$(ii)$, thus completing the induction step and the proof.
\end{proof}

\subsection{Proof of Claim~\ref{claim:k-reduction}}

\begin{proof}
%
	Put $G=G_H^k$, $\d'=2\sqrt{\d}$, and let $E \in E_k(\P)$ and $V \in V_k(\P)$.
	Note that $E$ is a $(k-1)$-partite $(k-1)$-graph, and let $(V_1,\ldots,V_{k-1})$ denote its vertex classes. Thus, $V_j \sub \Vside^j$ for every $1 \le j \le k-1$, and $V \sub \Vside^k$.
	Moreover, let $E' \sub E$, $V' \sub V$ with $|E'| \ge \d'|E|$, $|V'| \ge \d'|V|$.
	%
	To complete the proof our goal is to show that
	$d_{G}(E',V') \ge \frac12 d_{G}(E,V)$ (recall Definition~\ref{def:star-regular}).

%
	
	Consider the following $k$-partite $k$-graph and subgraph thereof:
	$$K = \{ (v_1,\ldots,v_k) \,\vert\, (v_1,\ldots,v_{k-1}) \in E \text{ and } v_k \in V \} = E \circ V \;,$$
	$$K' = \{ (v_1,\ldots,v_k) \,\vert\, (v_1,\ldots,v_{k-1}) \in E' \text{ and } v_k \in V' \} = E' \circ V' \;.$$
	We claim that
	\begin{equation}\label{eq:red-k-d}
	d_{G}(E,V) = \frac{|H \cap K|}{|K|} \quad\text{ and }\quad d_{G}(E',V') = \frac{|H \cap K'|}{|K'|} \;.
	\end{equation}
	Proving~(\ref{eq:red-k-d}) would mean that to complete the proof it suffices to show that
	\begin{equation}\label{eq:red-k-goal}
	\frac{|H \cap K'|}{|K'|} \ge \frac12 \frac{|H \cap K|}{|K|}  \;.
	\end{equation}
	To prove~(\ref{eq:red-k-d}) first note that
	\begin{equation}\label{eq:red-k-prod}
	|K| = |E||V| \quad\text{ and }\quad |K'|=|E'||V'| \;.
	\end{equation}
	Furthermore,
	\begin{align*}
	e_G(E,V) &= \big|\big\{ \, ((v_1,\ldots,v_{k-1}),v_k) \in G \,\big\vert\, (v_1,\ldots,v_{k-1}) \in E,\, v_k \in V \, \big\}\big| \\
	&= \big|\big\{ \, (v_1,\ldots,v_{k-1},v_k) \in H \,\big\vert\, (v_1,\ldots,v_{k-1}) \in E,\, v_k \in V \, \big\}\big|
	= |H \cap K| \;,
	\end{align*}
	and similarly, $e_G(E',V')=|H \cap K'|$.
	Therefore, using~(\ref{eq:red-k-prod}), we indeed have
	$$d_G(E,V) = \frac{e_G(E,V)}{|E||V|}
	= \frac{|H \cap K|}{|K|}
	\quad\text{ and }\quad d_G(E',V') = \frac{e_G(E',V')}{|E'||V'|}
	= \frac{|H \cap K'|}{|K'|} \;.$$
	
	Having completed the proof of~(\ref{eq:red-k-d}),
	it remains to prove~(\ref{eq:red-k-goal}).
	By Claim~\ref{claim:decomposition} there is a set of $k$-polyads $\{P_i\}_i$ of $\P$ on $(V_1,\ldots,V_{k-1},V)$ such that
	\begin{equation}\label{eq:red-k-partitionP}
	K = \bigcup_i \K(P_i) \,\text{ is a partition, with } P_i = (P_{i,1},\ldots,P_{i,k-1},E) \;.
	\end{equation}
	For each $k$-polyad $P_i=(P_{i,1},\ldots,P_{i,k-1},E)$, let $P_i'$ be the induced $k$-polyad $P_i' = P_i[V_1,\ldots,V_{k-1},V']$. Write $P_i'=(P'_{i,1},\ldots,P'_{i,k},E)$,
	and let $Q_i$ be the $k$-polyad $Q_i=(P'_{i,1},\ldots,P'_{i,k-1},E')$.
	Note that $Q_i$ satisfies $\K(Q_i) = \K(P_i) \cap K'$.
	It therefore follows from~(\ref{eq:red-k-partitionP}) that
	\begin{equation}\label{eq:red-k-partitionQ}
	K' = \bigcup_i (\K(P_i) \cap K') = \bigcup_i \K(Q_i) \,\text{ is a partition.}
	\end{equation}
	%
	%
	%
	%
	Suppose $\P$ is a $(k-1,a_1,a_2,\ldots,a_{k-1})$-partition, and denote $d_j = 1/a_j$ and
	$$d = \prod_{j=2}^{k-1} d_j^{\binom{k-1}{j-1}} \;.$$
	Put $\g = \frac18\d'$ ($\le \frac18$, as otherwise there is nothing to prove). We  will next apply the dense counting lemma (Fact~\ref{fact:counting}) to prove the estimates:
	\begin{equation}\label{eq:red-k-tP}
	|\K(P_i)| \le (1+\g)d|K|
	\end{equation}
	and
	\begin{equation}\label{eq:red-k-tQ}
	|\K(Q_i)| \ge \big(1-\g)d|K'| \;.
	\end{equation}
	Note that proving these estimates would in particular imply the bound
	\begin{equation}\label{eq:red-k-threshold}
	|\K(Q_i)| \ge \d|\K(P_i)| \;;
	\end{equation}
	indeed, from the assumptions $|E'| \ge \d'|E|$, $|V'| \ge \d'|V|$ and~(\ref{eq:red-k-prod}) we have that $|K'| \ge (\d')^2|K|$,
	hence we deduce from~(\ref{eq:red-k-tP}) and~(\ref{eq:red-k-tQ}) the lower bound
	$$
	\frac{|\K(Q_i)|}{|\K(P_i)|}
	\ge \frac{1-\g}{1+\g}(\d')^2
	\ge \frac34 \cdot (2\sqrt{\d})^2 \ge \d \;,
	$$
	where we used the inequality
	\begin{equation}\label{eq:red-k-quotient}
	\frac{1-\g}{1+\g} \ge 1-2\g \ge \frac34 \;.
	\end{equation}
	
	In order to prove~(\ref{eq:red-k-tP}) and~(\ref{eq:red-k-tQ})
	we first need to introduce some notation.
	Put $m=n/a_1$ where $n=|V(H)|$ is the size of the vertex set, and put
	$$
	\g' = \frac12\g \d'd \,\,\Big(=\frac14\d d\Big), \quad d_0 = \min\{d_2,\ldots,d_{k-1}\} \;.$$	
	%
	Note that $d \ge d_0^{2^k}$.
	Using the statement's assumption on $f$ we have (recall~(\ref{eq:DCL}))
	\begin{equation}\label{eq:red-k-f}
	f(d_0) = \d^4\Big(\frac{d_0}{2}\Big)^{2^{k+3}}
	\le \frac{\d}{2^{6}} \cdot \Big(\frac14\d d_0^{2^k}\Big)^3 \Big(\frac{d_0}{2}\Big)^{2^{k+1}}
	\le \frac{\d}{2}\cdot\frac{\g'^3}{12}\Big(\frac{d_0}{2}\Big)^{2^{k+1}} = \frac{\d}{2}F_{k,\g'}(d_0) \;.
	\end{equation}
	In particular,
	\begin{equation}\label{eq:red-k-eps}
	f(d_0) \le \g' d_0 \le \g' d_{k-1} \;.
	\end{equation}
	%
%
%
%
	%
%
%
	Note that if $P$ is an $\ell$-polyad of $\P$, for any $2 \le \ell \le k$, then, since $\P$ is $f$-equitable, the unique $\ell$-complex of $\P$ containing $P$ is $(f,d_2,\ldots,d_{\ell-1})$-equitable.
	Applying Fact~\ref{fact:counting} (dense counting lemma) with $\g'$ implies, using the fact that $f(x) \le F_{k,\g'}(x) \le F_{\ell,\g'}(x)$ by~(\ref{eq:red-k-f}) and the statement's assumption on $n$, that
	\begin{equation}\label{eq:red-k-count}
	|\K(P)| = (1 \pm \g')\prod_{j=2}^{\ell-1} d_j^{\binom{\ell}{j}} \cdot m^\ell \;.
	\end{equation}
	%

	%
	Let $P_E$ be the unique $(k-1)$-polyad of $\P$ such that $E \sub K(P_E)$.
	Then $|E|=d_E(P_E)|\K(P_E)|$, and since $\P$ is $f$-equitable, $E$ is $(d_{k-1},f(d_0))$-regular in $P_E$. In particular, $|E| \ge (d_{k-1}-f(d_0))|\K(P_E)|$.
	By~(\ref{eq:red-k-count}),
	$$|\K(P_E)| \ge (1-\g')\prod_{j=2}^{k-2} d_j^{\binom{k-1}{j}} \cdot m^{k-1} \;.$$
	Thus,
	\begin{equation}\label{eq:red-k-E}
	|E| \ge (d_{k-1}-f(d_0))|\K(P_E)| \ge (1-\g')d_{k-1}|\K(P_E)| \ge 
	(1-2\g')\prod_{j=2}^{k-1} d_j^{\binom{k-1}{j}} \cdot m^{k-1} \;,\
	\end{equation}
	where the second inequality uses~(\ref{eq:red-k-eps}).
	Furthermore, for every $P_i$ as above we have (recall $P_i$ is a $k$-polyad of $\P$), again by~(\ref{eq:red-k-count}), that
	%
	$$
	|\K(P_i)| \le (1+\g')\prod_{j=2}^{k-1} d_j^{\binom{k}{j}} \cdot m^k
	= (1+\g')d\prod_{j=2}^{k-1} d_j^{\binom{k-1}{j}} \cdot m^k
	\le \frac{1+\g'}{1-2\g'}d|E||V|
	\le (1+4\g')d|K| \;,
	$$
	where the penultimate inequality uses~(\ref{eq:red-k-E}), and the last inequality uses~(\ref{eq:red-k-prod}) and the fact that $\g' \le \g \le \frac18$.
	%
	This proves~(\ref{eq:red-k-tP}).
	
	
	Next we prove~(\ref{eq:red-k-tQ}). Let $\overline{P_i}$ be the unique $k$-complex of $\P$ containing the $k$-polyad $P_i$, and let $\overline{P_i}'$ be the induced $k$-complex $\overline{P_i}'=\overline{P_i}[V_1,\ldots,V_{k-1},V']$.
	Apply Lemma~\ref{lemma:k-slice} (slicing lemma) on $\overline{P_i}$, using the fact that $|V'| \ge \d'|V|$ and $f(x) \le \frac{\d}{2}F_{k-1,\frac14}$ by~(\ref{eq:red-k-f}), to deduce that $\overline{P_i}'$ is $(\frac{2}{\d}f,\, d_2,\ldots,d_{k-1})$-regular.
	Let the $k$-complex $\overline{Q_i}$ be obtained from the $k$-complex $\overline{P_i}'$ by replacing $E$ $(=P[V_1,\ldots,V_{k-1}])$ with $E'$,
	and note that the $(k-1)$-uniform hypergraph $\overline{Q_i}^{(k-1)}$ is precisely the $k$-polyad $Q_i$.
	%
	Apply Fact~\ref{fact:counting} (dense counting lemma, the ``moreover'' part) with $\g'$ on $\overline{P_i}'$, using the fact that 
	$\frac{2}{\d}f(x) 
	\le F_{k,\g'}(x)$ by~(\ref{eq:red-k-f}) and the statement's assumption on $n$, to deduce that
	$$
	|\K(Q_i)| \ge |E'| \cdot (1-\g')d|V'| - \g'|E|\cdot|V'|
	\ge \big(1-\g' - \frac{1}{\d' d}\g'\big)d|E'||V'|
	\ge \big(1-\g)d|K'| \;,
	$$
	where the second inequality uses the assumption that $|E'|\ge \d'|E|$ and the third inequality uses~(\ref{eq:red-k-prod}).
	This proves~(\ref{eq:red-k-tQ}).

	Finally, recall that our goal is to prove~(\ref{eq:red-k-goal}).
	We have
	\begin{align*}
	|H \cap K| &= \sum_i |H \cap \K(P_i)|
	= \sum_i d_H(P_i) \cdot |\K(P_i)|
	\le (1+\g)|K| \cdot d\sum_i d_H(P_i) \;,
	\end{align*}
	where the first equality uses~(\ref{eq:red-k-partitionP}) and the inequality is by~(\ref{eq:red-k-tP}).
	Denoting $d'=d\sum_i d_H(P_i)$, this means that
	\begin{equation}\label{eq:red-k-dP}
	\frac{|H \cap K|}{|K|} \le (1 + \g)d' \;.
	\end{equation}
	Observe that for every $i$, the statement's assumption on $\P$
	implies, together with~(\ref{eq:red-k-threshold}), that
	\begin{equation}\label{eq:red-k-reg}
	d_H(Q_i) \ge \frac23 d_H(P_i) \;.
	\end{equation}
	We have
	\begin{align*}
	|H \cap K'| &= \sum_i |H \cap \K(Q_i)| =
	\sum_i d_H(Q_i) \cdot |\K(Q_i)| \\
	&\ge \sum_i \frac23 d_H(P_i) \cdot |\K(Q_i)|
	\ge \frac23(1-\g)|K'| \cdot d\sum_i d_H(P_i) \;,
	\end{align*}
	where the first equality uses~(\ref{eq:red-k-partitionQ}), the first inequality uses~(\ref{eq:red-k-reg}) and the second inequality uses~(\ref{eq:red-k-tQ}).
	This means that
	$$\frac{|H \cap K'|}{|K'|} \ge \frac23(1-\g)d' \ge
	\frac23\cdot\frac{1-\g}{1+\g}\frac{|H \cap K|}{|K|} \ge \frac12 \frac{|H \cap K|}{|K|} \;,$$
	where the second inequality uses~(\ref{eq:red-k-dP}) and the third inequality uses~(\ref{eq:red-k-quotient}).
	We have thus proved~(\ref{eq:red-k-goal}) and are therefore done.
\end{proof}

%

\end{document}